\documentclass{amsart}
\usepackage{amsmath,amssymb,amsthm,geometry,graphics,color,mdwlist,mathrsfs}
\usepackage[all]{xy}

\DeclareMathOperator{\image}{Im}
\renewcommand{\Im}{\image}
\DeclareMathOperator{\Ker}{Ker}

\DeclareMathOperator*{\colim}{colim}

\DeclareMathOperator{\hocolim}{hocolim}
\DeclareMathOperator{\holim}{holim}
\DeclareMathOperator{\Hom}{Hom}
\DeclareMathOperator{\sHom}{\underline{\Hom}}
\DeclareMathOperator{\injHom}{\overline{\Hom}}
\DeclareMathOperator{\RHom}{\mathrm{R}\!\Hom}
\DeclareMathOperator{\cHom}{\mathscr{H}\hspace{-2pt}{\it om}}
\DeclareMathOperator{\End}{End}

\DeclareMathOperator{\REnd}{\mathrm{R}\!\End}

\def\lotimes{\otimes^\mathrm{L}}

\DeclareMathOperator{\rT}{\mathrm{T}}

\DeclareMathOperator{\Ext}{Ext}

\DeclareMathOperator{\pd}{proj.dim}

\DeclareMathOperator{\gldim}{gl.dim}

\DeclareMathOperator{\cone}{cone}

\DeclareMathOperator{\md}{mod}
\renewcommand{\mod}{\md}
\DeclareMathOperator{\Mod}{Mod}
\DeclareMathOperator{\smod}{\underline{\mod}}
\DeclareMathOperator{\injsmod}{\overline{\mod}}

\DeclareMathOperator{\proj}{proj}

\DeclareMathOperator{\coh}{coh}

\DeclareMathOperator{\res}{res}

\DeclareMathOperator{\sg}{sg}
\DeclareMathOperator{\cosg}{cosg}
\DeclareMathOperator{\defect}{def}

\DeclareMathOperator{\thick}{thick}
\DeclareMathOperator{\per}{per}
\DeclareMathOperator{\pvd}{pvd}
\DeclareMathOperator{\Loc}{Loc}

\DeclareMathOperator{\add}{add}
\DeclareMathOperator{\Add}{Add}

\DeclareMathOperator{\rad}{rad}

\DeclareMathOperator{\Spec}{Spec}

\DeclareMathOperator{\Tot}{Tot}

\def\rD{\mathrm{D}}

\def\rC{\mathrm{C}}

\def\Db{\mathrm{D^b}}
\def\Kb{\mathrm{K^b}}

\newcommand\recollement[3]{\xymatrix{{#1}\ar[r]&{#2}\ar[r]\ar@/_9pt/[l]\ar@/^9pt/[l]&{#3}\ar@/_9pt/[l]\ar@/^9pt/[l] }}
\newcommand\recollementwithmaps[9]{\xymatrix{{#1}\ar[r]|-{#8}&{#2}\ar[r]|-{#5}\ar@/_10pt/[l]_-{#7}\ar@/^10pt/[l]^-{#9}&{#3}\ar@/_10pt/[l]_-{#4}\ar@/^10pt/[l]^-{#6} }}
\def\A{\mathscr{A}}
\def\B{\mathscr{B}}
\def\C{\mathscr{C}}
\def\D{\mathscr{D}}

\def\F{\mathscr{F}}

\def\K{\mathscr{K}}

\def\M{\mathscr{M}}
\def\N{\mathscr{N}}

\def\P{\mathscr{P}}

\def\T{\mathscr{T}}
\def\U{\mathscr{U}}
\def\V{\mathscr{V}}

\def\Y{\mathscr{Y}}

\def\rL{\mathrm{L}}

\def\Ga{\Gamma}

\def\Del{\Delta}

\def\La{\Lambda}

\def\Si{\Sigma}

\def\Z{\mathbb{Z}}

\def\op{\mathrm{op}}
\def\dg{\mathrm{dg}}

\def\rsimeq{\rotatebox{-90}{$\simeq$}}
\def\xsimeq{\xrightarrow{\simeq}}

\def\vsubset{\rotatebox{90}{$\subset$}}

\def\romanenum{\renewcommand\labelenumi{(\roman{enumi})}\renewcommand\theenumi{\roman{enumi}}}

\newtheorem{Thm}{Theorem}[section]
\newtheorem{Lem}[Thm]{Lemma}
\newtheorem{Prop}[Thm]{Proposition}
\newtheorem{Cor}[Thm]{Corollary}

\newtheorem{Prop-Def}[Thm]{Proposition-Definition}
\newtheorem{Thm-Def}[Thm]{Theorem-Definition}

\theoremstyle{definition}
\newtheorem{Def}[Thm]{Definition}
\newtheorem{Ex}[Thm]{Example}

\theoremstyle{remark}
\newtheorem{Rem}[Thm]{Remark}

\newcounter{step}

\def\disoplus{\displaystyle\bigoplus}
\def\textoplus{\textstyle\bigoplus}

\def\Prod{\displaystyle\prod}

\makeatletter

\@addtoreset{equation}{section}
\makeatother

\title{Truncation of dg categories and connective resolutions}
\author{Norihiro Hanihara}
\thanks{This work is supported by JSPS KAKENHI Grant Numbers JP22KJ0737 and JP25K17233}
\address{Faculty of Mathematics, Kyushu University, 744 Motooka, Nishi-ku, Fukuoka, 819-0395, Japan}
\email{hanihara@math.kyushu-u.ac.jp}

\DeclareMathOperator{\srad}{sup.\!rad}

\date{}

\begin{document}
\begin{abstract}
We study the truncation $(-)^{\leq0}$ of dg categories. We first show that given a dg category $\C$ with shifts, for example a pretriangulated dg category, the canonical functor $\C^{\leq0}\to\C$ is a localization whose kernel is compactly generated by the $0$-th cohomology. 
Next we demonstrate that the derived category of the truncation serves as a triangulated analogue of the Auslander's category of coherent functors over abelian categories. We give a description of the Auslander-Reiten-Serre duality in terms of the inverse dualizing bimodule of the truncation.

Building on truncations, we introduce the notion of connective resolutions of dg categories, defined as a localization functor from a connective dg category. This notion encompasses non-commutative resolutions of algebras in module categories, cluster tilting objects in triangulated categories, and also the truncation as the universal connective resolution. We give a sufficient condition for an object in a triangulated category to give a connective resolution of its enhancement in terms of finiteness of resolution dimension.

Furthermore, we show that every proper dg module over a connective dg algebra is a direct summand of a dg module giving a connective resolution. Consequently, for every proper connective dg algebra, its bounded dg derived category has a connective resolution.

Finally, we give some explicit description of the truncation for some categories including the derived and cluster categories of a Dynkin quiver, and the Yoneda category of a finite dimensional algebra.
\end{abstract}

\maketitle
\setcounter{tocdepth}{1}
\tableofcontents
\section{Introduction}
The concept of non-commutative resolutions \cite{VdB04} is fundamental in interactions between singularity theory and representation theory.
Given a commutative ring $R$, its non-commutative resolution is an $R$-algebra $\Ga$ which has finite global dimension, and serves as an algebraic realization of the derived category of a (possibly imaginary) resolution of $\Spec R$, which therefore captures homological properties of the singularity.  
Among non-commutative resolutions, those {crepant}, namely non-commutative crepant resolutions (NCCRs) \cite{VdB04} are distinguished ones. Representation theoretically, these are precisely cluster tilting modules \cite{Iy07a,Iy07b} in the sense of higher dimensional Auslander-Reiten theory. They also play a crucial role in categorification of cluster algebras \cite{CA1, BMRRT}. 


Notice that these central concepts have a derived categorical nature in both algebraic geometry and in representation theory, and also in their interaction. Indeed, a non-commutative algebra models, or sometimes even becomes equivalent to, the derived category of a variety. Such a perspective appears in, for example, the derived McKay correspondence \cite{BKR,VdB04f}, Bondal-Orlov's conjecture on derived equivalence of (non-commutative) crepant resolutions \cite{BO95,VdB04}, homological minimal model program \cite{We18}, and so on.

In representation theory, the developments in tilting theory reveals the significance of silting objects \cite{KV,AI} which generalize ordinary algebras to derived ones, namely connective differential graded algebras. Also, the concept of cluster tilting objects \cite{Iy07a} which categorify clusters in cluster algebras are formulated in Calabi-Yau triangulated categories called cluster categories \cite{Am09,Ke11}, and the relationship between these kinds of objects is one of central structures in cluster tilting theory \cite{IYa1,HaI2}.
Furthermore, there are various recent developments to extend representation theory in ordinary module categories to the derived setting in terms of extriangulated categories \cite{NP,Wu23,Chen26,MoP}.

By such a derived categorical nature of non-commutative resolutions and cluster tilting objects, it is natural and essential to develop their derived formalism, which is the goal of this paper.
We introduce the notion of {\it connective resolutions} of differential graded (=dg) categories (\ref{def}). This unifies ordinary non-commutative resolutions in module categories, cluster tilting objects in triangulated categories, and so on.
For this we study the {\it truncation} of dg algebras and categories. 
We propose that the truncation serves as an enhancement of an additive subcategory of a triangulated subcategory, which enables us to reasonably formulate the finiteness of global dimension in the derived setting.

\subsection{Truncation of dg categories}
We first present the main construction which is elementary. Given a complex $X$, we denote by $X^{\leq0}$ the truncation at (cohomological) degree $0$, that is, a subcomplex
\[ X^{\leq0}=\left( \xymatrix{ \cdots\ar[r]&X^{-2}\ar[r]&X^{-1}\ar[r]&Z^0X\ar[r]&0\ar[r]&0\ar[r]&\cdots } \right) \]
of $X$, where $Z^0X$ is the $0$-cocycle. Given a dg category, the morphism-wise truncation yields the truncation of dg categories in the following sense.
\begin{Def}
Let $\C$ be a dg category. Its {\it truncation} $\C^{\leq0}$ is the dg category which has
\begin{itemize}
\item the same objects as $\C$, 
\item and the morphism complexes $\C^{\leq0}(A,B):=\C(A,B)^{\leq0}$ for $A,B\in\C^{\leq0}$.
\end{itemize}
There is a canonical functor $\C^{\leq0}\to\C$, which we also sometimes refer to as the truncation.
\end{Def}
Note that the truncation makes the dg category larger in a Morita theoretic sense. For example, the category $\per\C^{\leq0}$ has a silting subcategory $\add\{\C^{\leq0}(-,C)\mid C\in\C^{\leq0}\}$ whose summands are formed by the objects in the additive category $H^0\C$. 
In this way, each summand is viewed as a distinct object in the dg category $\C^{\leq0}$, while it may not be the case without truncating. 
Let us give a simple example illustrating this.
\begin{Ex}
Let $\C=\C^b_\dg(\proj k)$, the canonical dg enhancement of the derived category $\Db(\mod k)$ of a field $k$. Consider the object $M:=k\oplus k[1]\in\C$, or equivalently, the subcategory $\M=\add_\dg\!M\subset\C$. We compare $\M$ and its truncation $\M^{\leq0}$.
\begin{enumerate}
\item The endormphism dg algebra $\C(M,M)$ is (quasi-)isomorphic to the matrix algebra $\begin{pmatrix} k&k\\ k&k\end{pmatrix}$ with diagonal entries in degree $0$, upper right entry in degree $-1$, and lower left one in degree $1$. 
\item The truncated derived endomorphism algebra $\C(M,M)^{\leq0}$ is isomorphic to the upper triangular matrix algebra $\begin{pmatrix} k&k\\ 0&k\end{pmatrix}$ with diagonal entries in degree $0$ and the off-diagonal entry in degree $-1$. Thus, it is presented by the $A_2$-quiver with degree $-1$ arrow between the vertices: $\xymatrix{ \circ\ar[r]^-{-1}&\circ}$.
\end{enumerate}
This shows that while the ordinary dg endomorphism algebra $\C(M,M)$ is just derived Morita equivalent to $k$ and cannot distinguish the objects Morita theoretically, the {\it truncated} dg endomorphism algebra $\C(M,M)^{\leq0}$ can, for example, as the summmands of the silting object $\C(M,M)^{\leq0}\in\per\C(M,M)^{\leq0}$ in its derived category. 
\end{Ex}
This suggests that taking the truncation amounts to regarding a dg category as an ``underlying additive category'', in other words the category of projective modules over the Auslander algebra, with its derived information retained, which is a philosophy we will demonstrate in this paper.

Let us state the fundamental result of this paper, upon which the later concepts are motivated and build. Recall that a dg category $\C$ is {\it pretriangulated} if the Yoneda embedding $H^0\C\to\per\C$ is closed under mapping cones and cocones. Similarly, $\C$ {\it has shifts} if the Yoneda embedding is closed under both positive and negative shifts.
Also, a morphism $\A\to\B$ of dg categories is a {\it localization} if the induction functor $-\lotimes_\A\B$ induces an equivalence $\rD(\A)/\N\xsimeq\rD(\B)$ for some localizing subcategory $\N\subset\rD(\A)$.
\begin{Thm}[=\,\ref{loc}]\label{intro-loc}
Let $\C$ be a dg category with shifts, for example, pretriangulated.
\begin{enumerate}
\item The canonical morphism $\C^{\leq0}\to\C$ is a localization.
\item The kernel of $\rD(\C^{\leq0})\to\rD(\C)$ is compactly generated by $H^0\C$, which is compact in $\rD(\C^{\leq0})$.
\end{enumerate}
In other words, there is an exact sequence
\[ \xymatrix{ \thick_{\rD(\C^{\leq0})}H^0\C\ar[r]&\per\C^{\leq0}\ar[r]&\per\C }. \]
\end{Thm}

We give two accounts of this result.
First, from the spirit that the truncation $\C^{\leq0}$ is the underlying additive category of $\C$, its derived category $\per\C^{\leq0}$ must be the functor category \cite{Au66} of $\C$. We therefore call $\per\C^{\leq0}$ the {\it derived functor category} of $\C$. 
Let $\A$ be an abelian category, let $\mod\A$ be its functor category, that is, the category of finitely presented (contravariant) functors from $\A$ to $\Mod\Z$. 
Now, \ref{intro-loc} can be regarded as a derived analogue of so-called ``Auslander's formula'': 
if one denotes by $\defect\A$ the Serre subcategory of $\mod\A$ consisting of defects (=effaceable functors), then there is an exact sequence
\[ \xymatrix{ \defect\A\ar[r]&\mod\A\ar[r]&\A } \]
of abelian categories. One can observe that $\thick H^0\C$ is precisely the category of ``defects'' for triangles particularly when $\C$ is pretriangulated. We refer to \cite{Chen26} for this perspective.

Second, the exact sequence in \ref{intro-loc} shows that every algebraic triangulated category is a consingularity-type category. Recall that given a dg category $\A$, say connective of finite global dimension, its {\it cosingularity category} is the Verdier quotient 
\[ \cosg\A:=\per\A/\thick(\mod H^0\A) \]
of the perfect derived category by the ``bounded'' derived category.
It is Koszul dual to the construction of singularity categories \cite{Bu21}, and was introduced as the cluster category \cite{Am09} for Calabi-Yau dg algebras. 
The exact sequence stated in \ref{intro-loc} shows that a given triangulated category $\per\C$ is the ``cosingularity category'' of $\C^{\leq0}$ whose ``bounded'' derived category is precisely $\thick H^0\C$. We refer to \cite{Ha4,KY20,To}, also \ref{intro-conn} and \ref{intro-sg-cosg} below for such results realizing a given category as a cosingularity category.

\medskip

Motivated by the perspective that $\per\C^{\leq0}$ is the derived anagloue of the functor category, we give an account of the existence of (Auslander-Reiten-)Serre duality on a triangulated category in terms of derived functor categories along the line of an explantion of Auslander-Reiten(-Serre) duality on a module category in terms of their functor categories. 
More precisely, given a finite dimensional algebra $\La$, there is a commutative diagram
\[ \xymatrix@R=3mm{
	\mod(\mod\La)\ar[r]^-{\tau_2^{-1}}&\mod(\mod\La)\\
	\mod(\injsmod\La)\ar@{}[u]|-\vsubset\ar[r]^-{\tau^{-1}_\ast}&\mod(\smod\La)\ar@{}[u]|-\vsubset, } \]
where the first row is the inverse $\RHom_\M(D\M,-)[2]$ of the $2$-Auslander-Reiten translation for $\M=\mod\La$ on the functor category, and the second row is the functor induced by the stable equivalence $\tau^{-1}\colon\injsmod\La\xsimeq\smod\La$, see \ref{AR}. We prove, analogously, that the Auslander-Reiten translation on a triangulated category $\per\C$ is induced by certain $2$-Auslander-Reiten translation on the derived functor category $\per\C^{\leq0}$. Moreover, we relate such an existence of Auslander-Reiten-Serre duality with Amiot's construction \cite{Am09} of Serre functors on Verdier quotients in the derived sense \cite{HaI,HaK}. These results are summarized as follows.
\begin{Thm}[=\,\ref{DAR}]\label{intro-DAR}
Let $\C$ be a pretriangulated dg category. Suppose that the triangulated category $H^0\C$ has a Serre functor.
\begin{enumerate}
	\item\label{intro-theta} The inverse dualizing bimodule $\theta:=\RHom_{(\C^{\leq0})^e}(\C^{\leq0},(\C^{\leq0})^e)$ is invertible. In fact, there is an isomorphism $\theta[2](-,?)\simeq\C^{\leq0}(-,\tau^{-1}(?))$ in $\rD((\C^{\leq0})^e)$.
	\item\label{intro-cosg-type} There exists a canonical isomorphism $\C\lotimes_{\C^{\leq0}}(\theta[1])^{-1}\lotimes_{\C^{\leq0}}\C\xsimeq D\C$ in $\rD(\C^e)$.
\end{enumerate}
\end{Thm}
Here, the inverse $2$-Auslander-Reiten translation $\RHom_\M(D\M,-)[2]$ for the non-derived setting is replaced by the $2$-shifted inverse dualizing bimodule $\theta[2]$ of the truncation as in (\ref{intro-theta}). Also, the relationship between `$2$'-Auslander-Reiten translation for functor category and the `$1$'-Auslander-Reiten translation on the module category is explained as the `shift of degree $1$' appearing in Amiot's construction as in (\ref{intro-cosg-type}).

\medskip
We refer to \cite{Mo1,To} for results of Auslander correspondnces, and to Section \ref{rep} for some results related to the notion of representation dimension \cite{Au71,Iy03}. These are still other instances that the truncation should be regarded as the Auslander algebra.

\subsection{Connective resolutions}
We next introduce the notion of {\it connective resolutions} of dg categories, deriving that of non-commutative resolutions of ordinary algebras, and motivated by our fundamental theorem \ref{intro-loc}.
Recall that a {\it non-commutative resolution} of a Noetherian ring $\La$ is an algebra $\Ga$ of finite global dimension given as $\End_\La(M)$ for $M\in\mod\La$ with $\add\La\subset\add M$. We may allow several objects version and say that a subcategory $\M\subset\mod\La$ gives a non-commutative resolution if $\proj\La\subset\M$ and the functor category $\mod\M$ has finite global dimension. In this sense, the whole module category $\mod\La$ itself is a non-commutative resolution since the functor category has global dimension $\leq2$, and giving a non-commutative resolution amounts to choosing a nice object from the universal non-commutative resolution $\mod\La$.
This background, our localization theorem \ref{intro-loc}, together with the philosophy that $\per\C^{\leq0}$ is the functor category of $\C$, leads to the following definition.
\begin{Def}[=\,\ref{def}]\label{intro-def}
Let $\C$ be a pretriangulated dg category. A {\it connective resolution} of $\C$ is a subcategory $\M\subset\C$ satisfying the following two conditions.
\begin{enumerate}
\item The truncation $\M^{\leq0}$ has finite global dimension, in the sense of \cite{Ye13,Mi21}, see \ref{gldim}.
\item The composite $\M^{\leq0}\to\M\subset\C$ is a localization.
\end{enumerate}
Accordingly, an object $M\in\C$ is said to {\it give a connective resolution} if $\REnd_\C(M)^{\leq0}$ has finite global dimension and $\REnd_\C(M)^{\leq0}\to\C$ is a localization.
\end{Def}

Note that this notion of connective resolutions makes sense only by taking the truncation. Indeed, a usual non-commutative resolution $M$ in a module category $\mod\La$ is a generator of the derived category $\Db(\mod\La)$, so $\REnd_\La(M)$ is Morita equivalent to $\D^b_\dg(\mod\La)$ which is nothing new. However, taking the truncation yields $\REnd_\La(M)^{\leq0}=\End_\La(M)$ which is something different.  

Let us mention some immediate or motivating examples of connective resolutions.
First, if $\C$ is a pretriangulated dg category, then $\C$ itself gives a connective resolution $\C^{\leq0}$ of $\C$. Indeed, the truncation $\C^{\leq0}$ has global dimension $\leq2$ as is the case for ordinary functor categories (see \ref{DFun}) and the map $\C^{\leq0}\to\C$ is a localization by \ref{intro-loc}. From this perspective, giving a connective resolution amounts to choosing a nice subcategory or an object of the universal connective resolution $\C^{\leq0}$ which still localizes into $\C$.
We show that ordinary non-commutative resolution of a ring $\La$ in $\mod\La$ gives a connective resolution of $\D^b_\dg(\mod\La)$ (\ref{R}), and the canonical cluster tilting object in the cluster category gives its connective resolution (\ref{ct}). 
Also we observe that each silting object (\ref{silt}) gives a connective resolution of relevant dg categories. 
These examples indicate the generality of our notion of connective resolutions.
We refer to \ref{intro-conn} below for a general result for a subcategory to give a connective resolution.

Finally, we note that our notion of connective resolutions corresponds to a part of categorical resolution of singularities in the sense of Kuznetsov \cite{Ku08}. Recall that given a variety $X$, its categorical resolution is given by a diagram of (enhanced) triangulated categories
\[ \xymatrix@R=5mm{
	\per_\dg\!X\ar[r]^-{\pi^\ast}\ar@{=}[d]& \A\ar[d]^-{\pi_\ast}\\
	\per_\dg\!X\ar[r]& \D^b_\dg(\coh X) } \]
having suitable adjoints. Our connective resolution is a counterpart of the morphism $\pi_\ast\colon\A\to\D^b_\dg(\coh X)$. Indeed, for a commutative ring $R$ and $A=\End_R(M)$ is a non-commutative resolution, then $M\in\D^b_\dg(\mod R)$ gives both a categorical resolution of $\Spec R$ and a connective resolution of $\D^b_\dg(\mod R)$.


\medskip

We give a sufficient condition for a subcategory of a pretriangulated dg category to give a connective resolution, which can be formulated at the level of triangulated categories. Let $\T$ be a triangulated category and $\A\subset\T$ a full subcategory. A {\it right $\A$-resolution} of $X\in\T$ is a sequence of triangles
\[ \xymatrix{ X_{i+1}\ar[r]&A_i\ar[r]&X_i } \]
for $i\geq0$ with $X_0=X$, $A_i\to X_i$ a right $\A$-approximation (thus $A_i\in\A$). We say that $X\in\T$ {\it has finite right $\A$-resolution} if we may take $X_n\in\A$ for some $n\geq0$. We call $n$ the {\it length} of the resolution. 
\begin{Thm}\label{intro-conn}
Let $\C$ be a pretriangulated dg category and $\M\subset\C$ a full dg subcategory. Suppose that there is $n\geq0$ such that every object in $H^0\C$ has a finite right $H^0\M$-resolution of length at most $n$. 
\begin{enumerate}
	\item The subcategory $\M\subset\C$ gives a connective resolution, which has global dimension at most $n+2$.
	\item Then the kernel of the induction functor $-\lotimes_{\M^{\leq0}}\M\colon\per\M^{\leq0}\to\per\M=\per\C$ is generated by $\mod H^0\M$ as a thick subcategory.
\end{enumerate}
Therefore, there exists an exact sequence
\[ \xymatrix{ \thick_{\rD(\M^{\leq0})}(\mod H^0\M)\ar[r]&\per\M^{\leq0}\ar[r]&\per\C }. \]
\end{Thm}
This extends \ref{intro-loc} to the case of connective resolutions.
Consequently, such a connective resolution realizes a given triangulated category $\per\C$ as the consingularity category of $\M^{\leq0}$. This applies in particular to cluster tilting subcategories, and therefore each $d$-cluster tilting subcategory gives a connective resolution of global dimension $\leq d+1$. We refer to \ref{CT} for the statement.

\medskip
	
We next demonstrate the ubiquity of connective resolutions.
As a first such result, we give a derived analogue of Iyama's proof of finiteness of representation dimension \cite{Iy03}. He constructed a quasi-hereditary algebra by using certain chain of subcategories in a module category, called a rejective chain \cite{Iy05b}. We derive this construction by using a chain of subcategories in a derived category and show that a resulting dg algebra is ``quasi-hereditary'' (\ref{QH}). These invlove an instance of our philosophy that truncation gives an enhancement as an additive category.
\begin{Thm}[=\,\ref{repdim}]\label{intro-repdim}
Let $\La$ be a connective dg algebra. Then for every $M\in\Db(\La)$ there is $N\in\Db(\La)$ containing $M$ as a direct summand such that $\REnd_\La(N)^{\leq0}$ has finite global dimension.
\end{Thm}
Note that this is new even for finite dimensional algebras when $M$ is not necessarily a module: every complex of finitely generated module can be extended to a connective resolution, see \ref{cpx}. 

Applying \ref{intro-repdim} to $M=\La$ for a proper dg algebra $\La$, we obtain the following consequence.
\begin{Cor}[=\,\ref{NCR}]
For every connective dg algebra $\La$, its bounded dg derived category $\D^b_\dg(\La)$ has a connective resolution.
\end{Cor}
As a further consequence, we obtain that every singularity category is a cosingularity category.
\begin{Cor}[=\,\ref{sg-cosg}]\label{intro-sg-cosg}
Let $\La$ be a proper connective dg algebra. Then there exists a connective dg algebra $\Si$ of finite global dimensionsuch that $\sg_\dg\!\La\simeq\cosg_\dg\!\La$.
\end{Cor}

As another class of examples of connective resolutions, we discuss the truncation of the {\it Yoneda category} \cite{Ha2}. Let $\La$ be a finite dimensional algebra, and consider the dg category
\[ \Y=\add_\dg\{ A[i]\mid A\in\mod\La, i\in\Z\}\subset\D^b_\dg(\mod\La). \]
It is an enhancement of the Yoneda category which is a categorical counterpart of the Yoneda algebra $\bigoplus_{i\geq0}\Ext^i_\La(M,M)$ when $\La$ is representation finite and $M\in\mod\La$ is an additive generator \cite{Ha2}. Therefore, its truncation $\Y^{\leq0}$ is what should be called the derived Yoneda algebra.
Note that the $\Y$ is a dg category with shifts. Our localization theorem \ref{intro-loc} leads to the following result.
\begin{Thm}[=\,\ref{yoneda-ncr}]
	Let $\La$ be a finite dimensional algebra. Then $\Y\subset\D^b_\dg(\mod\La)$ gives a connective resolution of global dimension at most $3$.
\end{Thm}

\subsection{Examples}
We give explicit descriptions of the truncation. We consider the dg categories related to derived categories of a Dynkin quiver. It is well-known that for a Dynkin quiver $Q$ the Auslander-Reiten quiver of the derived category $\Db(\mod kQ)$ is the infinite translation quiver $\Z Q$ \cite{Hap}, and it is standard \cite{Rie,Hap}, that is, the category $\Db(\mod kQ)$ is presented by the Auslander-Reiten quiver with the mesh relations. We consider the {derived mesh} (or more precisely derived preprojective) category as follows. It is a dg category defined by adding the degree $-1$ arrows to the direction of $\tau^{-1}$, which differentiates into the preprojective relation (\ref{dpp}). For type $A_4$ it looks as below.
\[ \xymatrix@!R=1mm@!C=3mm{
	\ar[r]&\circ\ar[rr]^-{-1}\ar[dr]&&\circ\ar[rr]^-{-1}\ar[dr]&&\circ\ar[rr]^-{-1}\ar[dr]&&\circ\ar[rr]^-{-1}\ar[dr]&&\circ\ar[rr]^-{-1}\ar[dr]&&\circ\\
	\circ\ar[ur]\ar[rr]^-{-1}\ar[dr]&&\circ\ar[ur]\ar[rr]^-{-1}\ar[dr]&&\circ\ar[ur]\ar[rr]^-{-1}\ar[dr]&&\circ\ar[ur]\ar[rr]^-{-1}\ar[dr]&&\circ\ar[ur]\ar[rr]^-{-1}\ar[dr]&&\circ\ar[ur]\ar[dr]\ar[r]&\\
	\ar[r]&\circ\ar[rr]^-{-1}\ar[dr]\ar[ur]&&\circ\ar[rr]^-{-1}\ar[dr]\ar[ur]&&\circ\ar[rr]^-{-1}\ar[dr]\ar[ur]&&\circ\ar[rr]^-{-1}\ar[dr]\ar[ur]&&\circ\ar[rr]^-{-1}\ar[dr]\ar[ur]&&\circ\\
	\circ\ar[ur]\ar[rr]^-{-1}&&\circ\ar[ur]\ar[rr]^-{-1}&&\circ\ar[ur]\ar[rr]^-{-1}&&\circ\ar[ur]\ar[rr]^-{-1}&&\circ\ar[ur]\ar[rr]^-{-1}&&\circ\ar[ur]\ar[r]&  } \]
From the idea that the derived category of an algebra is the ``graded cluster category'' \cite[A.7]{Ha7}, we obtain the description of the truncation of $\D^b_\dg(\mod kQ)$ as the derived preprojective category.
\begin{Thm}[=\,\ref{mesh}]\label{inrto-mesh}
Let $Q$ be a Dynkin quiver. Then the derived preprojective category of $Q$ is quasi-equivalent to the full subcategory of the truncation $\D^b_\dg(\mod kQ)^{\leq0}$ formed by indecomposable objects.
\end{Thm}
By considering its subcategories, it easily leads to the description of the truncated derived endomorphism algebra $\REnd_{kQ}(M)^{\leq0}$ for each $M\in\Db(\mod kQ)$, see for example \ref{2term}.
We also observe (\ref{orbit}) that truncation is compatible with taking the orbit categories. This leads to the description of, for example, the truncation of dg cluster categories of Dynkin quivers, see \ref{cluster}.

\subsection*{Acknowledgements}
Some of the results in this paper, including \ref{loc}, were presented at informal seminars in Paris (August 2022) and Cologne (January 2025). The author would like to thank Bernhard Keller and Gustavo Jasso for valuable discussions, which lead to the interpretation of truncation as the derived additivization. The author would also like to thank Osamu Iyama and Ryu Tomonaga for stimulating discussions.

\section{Preliminaries}
\subsection{Conventions on truncations}
For a complex $X=(\cdots\to X^{i-1}\xrightarrow{d} X^i\xrightarrow{d} X^{i+1}\to\cdots)$ and $i\in\Z$, we denote by
\[ X^{\leq i}=( \cdots\to X^{i-1}\xrightarrow{d} Z^iX\xrightarrow{} 0\to\cdots) \]
the truncated complex, where $Z^iX=\Ker d$ is the $i$-cocycle. This is a subcomplex of $X$ and we denote by $X^{>i}$ the quotient $X/X^{\leq i}$. 

If $A$ is a dg algebra, then the truncation $A^{\leq0}$ is a dg subalgebra, and each $A^{\leq i}$ is a dg (bi)module over $A^{\leq0}$. More generally, if $\A$ is a dg category the morphismwise truncation $\A^{\leq0}$ yields a dg category as follows.
\begin{Def}
Let $\A$ be a dg category. Its {\it truncation} $\A^{\leq0}$ is the dg category with
\begin{itemize}
\item the same objects as $\A$,
\item and with the morphism complex $\A^{\leq0}(A,B):=\A(A,B)^{\leq0}$ for each pair of objects $A$ and $B$.
\end{itemize}
\end{Def}
We have a canonical functor $\A^{\leq0}\to\A$, which we sometimes refer to as the truncation morphism. For each $i\in\Z$, the we denote by $\A^{\leq i}$ the dg $(\A^{\leq0},\A^{\leq0})$-bimodule $(A,B)\mapsto\A(A,B)^{\leq i}$, thus $\A^{\leq i}(A,B)=\A(A,B)^{\leq i}$.

We use the following two easy facts on complexes:
\begin{enumerate}
	\item For each complex $X$, there is a canonical isomorphism $\hocolim_{p\geq1}X^{\leq p}\xsimeq X$.
	\item For each complex $X$, there is a canonical isomorphism $X\xsimeq\holim_{p\geq1}X^{\geq -p}$.
\end{enumerate}

\subsection{Connective dg categories and their global dimension}
We start with the fundamental notion which will be used throughout this papar.
\begin{Def}
A dg category $\A$ is called {\it connective} if $H^i\A(A,B)=0$ for every $A,B\in\A$ and $i>0$.
\end{Def}
If $\A$ is a connective dg category, then the subcategory $\add\{\A(-,A)\mid A\in\A\}$ is a {silting subcategory} \cite{KV,AI} of the perfect derived category $\per\A$.

Following \cite{Ye13,Mi21}, we review the notion of global dimension of connective dg categories.
Let $\A$ be a connective dg category. Then its derived category $\rD(\A)$ has a standard $t$-structure $(T^{\leq0},T^{\geq0})$ given by
\[ 
\begin{aligned}
	T^{\leq0}&=\{X\in\rD(\A)\mid H^iX=0 \text{ for all } i>0 \},\\
	T^{\geq0}&=\{X\in\rD(\A)\mid H^iX=0 \text{ for all } i<0 \},
\end{aligned}
\]
and its heart is equivalent to $\Mod H^0\A$.
Let $\P=\Add\{\A(-,A)\mid A\in\A\}\subset\rD(\A)$ be the subcategory of representable modules.
We say that a dg $\A$-module $M$ is {\it uniformly bounded} (resp. {\it uniformly bounded above} etc.) if there is an interval $[a,b]\subset\Z$ (resp. $a\in\Z$ etc.) such that for all $A\in\A$ we have $MA\in\rD^{[a,b]}(\Mod k)$ (resp. $MA\in\rD^{[-\infty,a]}(\Mod k)$ etc). We denote by $\rD^{u,-}(\A)$ the full subcategory of $\rD(\A)$ consiting of objects which are uniformly bounded below. For $M\in\rD^{u,-}(\A)$ we set $\sup M=\sup\{n\in\Z \mid H^n(MA)\neq0 \text{ for some }A\in\A\}$.

For subcategories $\U$ and $\V$ of a triangulated category $\T$, we put
\[ \U\ast\V=\{ X\in\T\mid \text{there is a triangle } U\to X\to V \text{ for some } U\in\U \text{ and } V\in\V\}. \]
\begin{Def}[\cite{Ye13,Mi21}]\label{PD}
Let $\A$ be a connective dg category. Let $M\in\rD^{u,-}(\A)$ with $\sup M=n$. The {\it projective dimension} of $M$ is the minimum $d\in\Z_{\geq0}$ satisfying the following equivalent conditions.
\begin{itemize}
\item $\RHom_\A(M,-)$ takes $\Mod H^0\A$ to $\rD^{[-n,-n+d]}(k)$.
\item $M\in\P[-n]\ast\P[-n+1]\ast\cdots\P[-n+d]$.
\end{itemize}
\end{Def}

For our purpose, we only consider projective dimensions of finitely presented modules over ``coherent'' dg categories.
\begin{Lem}\label{pds}
Let $\A$ be a connective dg category such that $H^0\A$ is right coherent and each $H^i\A$ is finitely presented as a right $H^0\A$-module. For $M\in\mod H^0\A$, the following are equivalent.
\begin{enumerate}
\romanenum
\item\label{largepd} $\pd M\leq n$ in the sense of \ref{PD}.
\item\label{smallpd} $M\in(\add\A)\ast\cdots\ast(\add\A[n])$.
\end{enumerate}
\end{Lem}
\begin{proof}
	The implication (\ref{smallpd})$\Rightarrow$(\ref{largepd}) is clear, so we prove the converse. Assume that $\pd M\leq n$. By the coherence assumptions on $\A$ and since $M\in\mod H^0\A$, there are triangles $M_{i+1}\to P_i\to M_i$ with $M_0=M$, $P_i\in\add\A$ and $H^0P_i\to H^0M_i$ is surjective. It follows from \cite[2.13]{Mi21} that $M_l$ must have projective dimension $0$ for some $l\leq n$. By the proof of \cite[2.14]{Mi21}, such $M_l$ is a direct summand of $P_l$, and hence in $\add\A$. This proves (\ref{smallpd}).
\end{proof}
	
We use the following variant(!) of the notion of global dimension of \cite{Mi21} for connective dg categories. This is called weak global dimension in \cite{To}.
\begin{Def}\label{gldim}
Let $\A$ be a connective dg category. The (right) {\it global dimension} of $\A$ is the supremum of projective dimensions of objects in the category of finitely presented modules $\mod H^0\A\subset\rD(\A)$.
\[ \gldim\A=\sup\{\pd M\mid M\in\mod H^0\A\}. \]
\end{Def}
If $H^0\A$ is (left and right) coherent and $H^i\A$ is finitely presented on each side, then by \cite{To} the above notion is left-right symmetric: $\gldim\A=\gldim\A^\op$.

We next show that finiteness of global dimension is preserved under dg quotients, compare \cite[3.10]{Ke11}.
\begin{Prop}\label{gldimA/P}
	Let $\P$ be a connective subcategory of a connective dg category $\A$. Then $\A/\P$ is connective. Also, if $H^0\P\subset H^0\A$ is contravariantly finite, then $\gldim\A/\P\leq\gldim\A$.
\end{Prop}
\begin{proof}
	By the triangle $\A\lotimes_\P\A\to\A\to\A/\P$, we see that the dg quotient $\A/\P$ is connective. The same triangle shows that the canonical map $H^0\A\to H^0(\A/\P)$ is surjective (on each morphism space), and that $H^0(\A/\P)(-,A)$ is finitely presented for each $A\in\A$ as an $H^0\A$-module since $H^0\P\subset H^0\A$ is contravariantly finite.
	Since $\A\to\A/\P$ is a dg quotient, the restriction functor $\rD(\A/\P)\to\rD(\A)$ is fully faithful, and it restricts to the functor between the hearts $\Mod H^0(\A/\P)\to\Mod H^0\A$, and therefore to $\mod H^0(\A/\P)\to\mod H^0\A$ since $H^0(\A/\P)(-,A)\in\mod H^0\A$. 
	It follows that for each $X,Y\in\mod H^0(\A/\P)$, we have $\RHom_{\A/\P}(X,Y)\xsimeq\RHom_\A(X,Y)$ is concentrated in degree $\leq\gldim\A$. Therefore, we deduce that $\gldim\A/\P\leq\gldim\A$, as desired.
\end{proof} 

We next observe that finiteness of global dimension is preserved under triangular extensions.
Let $\A$ and $\B$ be dg categories, let $M$ be a $(\B,\A)$-bimodule. Then, let $\C$ be the dg category whose
\begin{itemize}
\item objects are disjoint union of objects of $\A$ and $\B$,
\item morphisms are given by $\C(U,V)=\begin{cases} \A(U,V)& \text{ if } U,V\in\A \\ M(U,V) & \text{ if } U\in\A \text{ and } V\in\B \\ 0& \text{ if } U\in\B \text{ and } V\in\A \\ \B(U,V) & \text{ if } U,V\in\B\end{cases}$.
\end{itemize}
We denote this category by $\C=\begin{pmatrix}\A&0\\ M&\B\end{pmatrix}$.
Then a dg $\C$-module $Z$ is given by a triple $(X,Y;f)$, where $X$ is an $\A$-module, $Y$ is a $\B$-module, and $f$ is an $\A$-linear map $Y\lotimes_\B M\to X$. 
\begin{Lem}\label{ext}
Let $Z=(X,Y;f)\in\rD(\C)$ and $Z^\prime=(X^\prime,Y^\prime;f^\prime)\in\rD(\C)$, and denote $N=\cone f\in\rD(\A)$. Then there is a triangle
	\[ \xymatrix{ \RHom_\A(N,X^\prime)\ar[r]&\RHom_\C(Z,Z^\prime)\ar[r]&\RHom_\B(Y,Y^\prime) }. \]
\end{Lem}
\begin{proof}
	This is shown in a similar way as in \cite[6.3]{BHMPT}.
\end{proof}
	
\begin{Prop}\label{SOD}
Let $\A$ and $\B$ be connective dg categories, let $M$ be a $(\B,\A)$-bimodule with $\sup M\leq0$, and let $\C=\begin{pmatrix}\A&0\\ M&\B\end{pmatrix}$. Then there are inequalities
\[ \sup\{\gldim\A,\gldim\B\}\leq \gldim\C\leq \sup\{ \gldim\A, \pd_\A\! M+\gldim\B+1\}. \]
\end{Prop}
\begin{proof}
	Note that $\A$ is a dg quotient of $\C$ by $\B$, and similarly $\B$ is a dg quotient of $\C$ by $\A$, thus the first inequality follows from \ref{gldimA/P}.
	
	To prove the second one, we note some easy obvervations.
	\begin{enumerate}
	\item If $L_1\to L_2\to L_3$ is a triangle, then $\pd L_3\leq \sup\{\pd L_2,\pd L_1+1\}$.
	\item We have $\pd_\A(Y\lotimes_\B M)\leq \pd_\B\!Y+\pd_\A\!M$.
	\end{enumerate}
	Indeed, the first assertion is clear, and the second one follows from the adjunction $\RHom_\A(Y\lotimes_\B M,-)=\RHom_\B(Y,\RHom_\A(M,-))$.
	
	Now, let $Z=(X,Y;f)$ present an $H^0\C$-module, in which case $X$ is an $H^0\A$-module and $Y$ is an $H^0\B$-module. Let $N=\cone(Y\lotimes_\B M\xrightarrow{f} X)\in\rD(\A)$ as above. Then the desired inequality follows as
	\[
	\begin{aligned}
		\gldim\C&\leq \sup\{\pd N,\gldim\B\} && \text{(by \ref{ext})} \\
		&\leq \sup\{ \pd X,\pd(Y\lotimes_\B M)+1,\gldim\B\} && \text{(by (1))}\\
		&\leq \sup\{ \gldim\A, \pd_\A\! M+\pd_\B\! Y+1,\gldim\B\}&& \text{(by (2))}\\
		&\leq\sup\{ \gldim\A, \pd_\A\! M+\gldim\B+1\}.&&\qquad\qquad\qedhere 
	\end{aligned}
	\]
\end{proof}

Finally, we discuss the relationship between another possible notion of finiteness of global dimension. Recall that a dg category is {\it smooth} if it is perfect as a bimodule.
\begin{Prop}\label{smooth}
Each smooth connective dg category has finite global dimension.
\end{Prop}
\begin{proof}
	Let $\A$ be a smooth dg category, that is, $\A\in\per\A^e$. If $\A$ is connective, then so is $\A^e$, and since $\sup\A=0$, we have $\A\in\A^e\ast\A^e[1]\ast\cdots\ast\A^e[n]$ for some $n\geq0$. Then for each $X\in\Mod H^0\A$, applying $X\lotimes_\A-$ yields $X\in\A\ast\A[1]\ast\cdots\ast\A[n]$, thus $\pd X\leq n$.
\end{proof}

We also prove the converse under certain finiteness assumptions.
\begin{Prop}\label{converse}
Let $A$ be a connective dg algebra over a field $k$. Suppose that each cohomology is finite dimensional and $H^0A/\rad H^0A$ is separable over $k$. If $A$ has finite global dimension then $A$ is smooth.
\end{Prop}

We need one preparation.
\begin{Lem}[{c.f.\,\cite[4.2.2]{KeL}}]\label{pd}
Let $A$ be a connective dg algebra such that each cohomology is finite dimensional and let $S$ be the direct sum of all simple $H^0A$-modules. Let $M$ be a dg module which is bounded above and each cohomology is finite dimensional. Then $M\in\per A$ if and only if $\RHom_A(M,S)\in\per k$.
\end{Lem}
\begin{proof}
	The only if part is clear, so we prove the if part. We may assume that $M=M^{\leq0}$. By our assumption that each cohomology of $M$ is finite dimensional, we may take a `minimal projective resolution', that is, a sequence of triangles
	\[ \xymatrix@R=2mm{ M_{i+1}\ar[r]& P_i\ar[r]& M_i } \]
	for $i\geq0$ with $M_0=M$ such that $P_i\to M_i$ induces a projective cover $H^0P_i\to H^0M_i$ in $\mod H^0A$.
	Note that this implies that each map $P_i\to P_{i-1}$ a radical map in $\per A$. Also, if we put $Q_n=\Tot(P_n\to\cdots P_0)$, the canonical map $\hocolim_{n\geq0}Q_n\to M$ is an isomorphism in $\rD(A)$.
	Then the complex $\RHom_A(M,S)$ is computed as the homotopy limit of the sequence
	\[ \xymatrix{ \cHom_A(P_0,S)\ar[r]&\cHom_A(P_1,S)\ar[r]&\cdots }, \]
	in which every transition map is $0$ since the resolution is minimal and $S$ is semisimple. It follows that $\RHom_A(M,S)=\prod_{n\geq0}\cHom_A(P_n,S)[-n]$, which has a non-zero cohomology in degree $\geq n$ whenever $P_n\neq0$. We conclude that the boundedness of $\RHom_A(M,S)$ forces $P_n=0$ for $n\gg0$, thus $M\in\per A$.
\end{proof}

We are now ready to prove \ref{converse}.
\begin{proof}[Proof of \ref{converse}]
	Let $S$ be the direct sum of all the simple $H^0A$-modules. Then the separability assumption shows that $DS\otimes S=\cHom_k(S,S)$ is the direct sum of all the simple $H^0A^e$-modules. Now the complex $\RHom_{A^e}(A,\cHom_k(S,S))=\RHom_A(S,S)$ is bounded by $\gldim A<\infty$, thus applying \ref{pd} to the $A^e$-module $A$ yields that $A\in\per A^e$, that is, $A$ is smooth.
\end{proof}

We end this preliminary section with the following descriptions of inverse dualizing bimodules.
\begin{Prop}\label{vee}
Let $\A$ be a dg category such that each $H^i\A$ is finite dimensional and for each $X,Y\in\A$ one has $H^i\A(X,Y)=0$ for $i\gg0$. Then we have $\RHom_{\A^e}(\A,\A^e)=\RHom_\A(D\A,\A)$ in $\rD(\A^e)$.
\end{Prop}
\begin{proof}
	By assumptions on $\A$ we have quasi-isomorphisms $\cHom_k(D\A(X,-),\A(-,Y))\leftarrow DD\A(X,-)\otimes\A(-,Y)\leftarrow\A(X,-)\otimes\A(-,Y)$. In particular we see that $\cHom_k(D\A,\A)\simeq\A^e$ in $\rD(\A^e)$. Therefore we conclude $\RHom_\A(D\A,\A)=\RHom_{\A^e}(\A,\cHom_k(D\A,\A))=\RHom_{\A^e}(\A,\A^e)$.
\end{proof}

\section{Truncation of dg categories}
\subsection{Truncation is a localization}
We start our discussion with the following fundamental result, on which the later results are based. We say that a dg category $\A$ {\it has shifts} if $\A(-,A)[\pm1]\in\rD(\A)$ is isomorphic to a representable $\A$-module. For example, each pretriangulated dg category has shifts.
Recall also that a morphism $\A\to\B$ of dg category is a {\it localization} if the induction functor $-\lotimes_\A\B$ induces an equivalence $\rD(\A)/\N\xsimeq\rD(\B)$ for some localizing subcategory $\N\subset\rD(\A)$. This is the case if and only if the restriction functor $\rD(\B)\to\rD(\A)$ is fully faithful.
\begin{Thm}\label{loc}
Let $\C$ be a dg category with shifts.
\begin{enumerate}
\item The canonical map $\C^{\leq0}\to\C$ is a localization.
\item The kernel of the induction functor $\rD(\C^{\leq0})\to\rD(\C)$ is compactly generated by $H^0\C$, which is compact in $\rD(\C^{\leq0})$. 
\end{enumerate}
Therefore, we have an equivalence $-\lotimes_{\C^{\leq0}}\C\colon\per\C^{\leq0}/\thick H^0\C\xsimeq\per\C$.
\end{Thm}
%

We start the proof of \ref{loc} with an easy observation on modules over dg categrories with shifts.
\begin{Lem}\label{resol}
For each $X\in\rD(\C)$, $H^0X=0$ implies $X=0$.
\end{Lem}
\begin{proof}
	For any $C\in\C$ and $i\in\Z$, we have $(H^iX)(C)=\Hom_{\rD(\C)}(\C(-,C[-i]),X)=(H^0X)(C[-i])=0$.
\end{proof}

\begin{Lem}\label{per}
	For each $C\in\C$ and $i\in\Z$, the $\C^{\leq0}$-module $H^i\C(-,C)$ is perfect.
\end{Lem}
\begin{proof}
	Since $H^i\C(-,C)=H^0\C(-,C[i])$, we may assume $i=0$. We have a triangle
	\[ \xymatrix{
		\C(-,C)^{\leq-1}\ar[r]&\C^{\leq0}(-,C)\ar[r]&H^0\C(-,C) } \]
	in $\rD(\C^{\leq0})$. Note that $\C(-,C)^{\leq-1}=(\C(-,C)[-1])^{\leq0}[1]=\C^{\leq0}(-,C[-1])[1]$, which is hence perfect. It follows that $H^0\C$ is also perfect.
\end{proof}

Recall that a dg $\A$-module $M$ is {\it uniformly bounded below} if there is $a\in\Z$ such that for all $A\in\A$ we have $MA\in\rD^{[a,\infty]}(\Mod k)$. We denote by $\rD^{u,+}(\A)$ the full subcategory of $\rD(\A)$ consiting of objects which are uniformly bounded below.
\begin{Lem}\label{ubb}
	Let $M\in\rD^{u,+}(\C^{\leq0})$. Then $M\lotimes_{\C^{\leq0}}\C=0$.
\end{Lem}
\begin{proof}
	Since $M$ is the homotopy colimit of uniformly bounded $\C^{\leq0}$-modules $M^{\leq p}$, we may assume $M$ is concentrated in degree $0$. Then $M$ is in the heart $\Mod H^0\C$ of the standard $t$-structure on $\rD(\C^{\leq0})$. Since the restriction functor $\rD(H^0\C)\to\rD(\C^{\leq0})$ preserves coproducts, $M$ lies in the localizing subcategory of $\rD(\C^{\leq0})$ generated by $H^0\C$. Therefore it is enough to prove the lemma for $M=H^0\C$, in particular we may assume $M$ is perfect. 
	
	Applying $M\lotimes_{\C^{\leq0}}-$ to the triangle $\C^{\leq0}\to\C\to\C^{>0}$, we have a triangle in $\rD(\C^{\leq0})$:
	\[ \xymatrix{	M\ar[r]&M\lotimes_{\C^{\leq0}}\C\ar[r]&M\lotimes_{\C^{\leq0}}\C^{>0} }. \]
	We have that $M$ is uniformly bounded below by assumption, and so is the last term since $M\in\per\C^{\leq0}$ and $\C^{>0}\in\rD^{u,+}(\C^{\leq0})$, It follows that the middle term $M\lotimes_{\C^{\leq0}}\C$ is also uniformly bounded below, which shows it is $0$ by \ref{resol}.
\end{proof}

Let us note another computation.
\begin{Lem}\label{ast}
	We have an isomorphism $\RHom_{\C^{\leq0}}(H^0\C,\C^{\leq0})=H^1\C[-2]$ in $\rD((\C^{\leq0})^e)$.
\end{Lem}
\begin{proof}
	Consider the triangle
	\[ \xymatrix{
		\C^{\leq-1}(-,?)\ar[r]&\C^{\leq0}(-,?)\ar[r]&H^0\C(-,?) }. \]
	Note that we have $\C^{\leq-1}=\C^{\leq0}(-,?[-1])[1]$, thus applying $\RHom_{\C^{\leq0}}(-,\C^{\leq0})$ yields a triangle
	\[ \xymatrix{ \RHom_{\C^{\leq0}}(H^0\C(-,?),\C^{\leq0})\ar[r]&\C^{\leq0}(?,-)\ar[r]&\C^{\leq0}(?[-1],-)[-1] }. \]
	Now, similarly as above, we have $\C^{\leq0}(?[-1],-)[-1]=\C(?,-)^{\leq1}$, thus the mapping cocone of the second term is $H^1\C(?,-)[-2]$.
\end{proof}

Now we are ready to prove the main result of this section.
\begin{proof}[Proof of \ref{loc}]
	(1)  We show that the natural map $\C\lotimes_{\C^{\leq0}}\C\to\C$ is an isomorphism in $\rD(\C^e)$. Applying $-\lotimes_{\C^{\leq0}}\C$ yields to the triangle $\C^{\leq0}\to\C\to\C^{>0}$ yields
	\[ \xymatrix{ \C\ar[r]&\C\lotimes_{\C^{\leq0}}\C\ar[r]&\C^{>0}\lotimes_{\C^{\leq0}}\C }. \]
	By \ref{ubb} the third term is $0$, and thus the first map is an isomorphism. Since the first map is a right inverse to the multiplication map $\C\lotimes_{\C^{\leq0}}\C\to\C$, the assertion follows.
	
	(2)  We prove that there exists a semi-orthogonal decomposition $\rD(\C^{\leq0})=\Loc H^0\C\perp\rD(\C)$.
	
	We first show $\RHom_{\C^{\leq0}}(H^0\C,\C)=0$. Since $H^0\C\in\per\C^{\leq0}$ by \ref{per}, we have $\RHom_{\C^{\leq0}}(H^0\C,\C)=\C\lotimes_{\C^{\leq0}}\RHom_{\C^{\leq0}}(H^0\C,\C^{\leq0})$, which is isomorphic to $\C\lotimes_{\C^{\leq0}}H^1\C[-2]$ by \ref{ast}. We see that this is $0$ by (the opposite of) \ref{ubb}.
	
	We next prove that any object in $\rD(\C^{\leq0})$ breaks into a suitable triangle. Let $X\in\rD(\C^{\leq0})$. Applying $X\lotimes_{\C^{\leq0}}-$ to the triangle $\C^{\leq0}\to\C\to\C^{>0}$ gives a triangle
	\[ \xymatrix{
		X\lotimes_{\C^{\leq0}}\C^{>0}[-1]\ar[r]&X\ar[r]& X\lotimes_{\C^{\leq0}}\C }, \]
	in $\rD(\C^{\leq0})$. The right term $X\lotimes_{\C^{\leq0}}\C$ is clearly in $\rD(\C)$. On the other hand, the leftmost term is in the localizing subcategory generated by $H^0\C$ since $\C^{>0}$ is a homotopy colimit of the $H^i\C$'s. This gives our assertion.
\end{proof}

We give a quick example of the truncation $\C^{\leq0}$. Further examples shall be found in Section \ref{Examples}.
\begin{Ex}
Let $\La$ be a dg $k$-algebra and $\C$ the full subcategory of $\per_\dg\!\La$ consisting of the shifts of $\La$. Then $\C^{\leq0}$ is the category with
\begin{itemize}
	\item objects: $\Z$, with $i\in\Z$ corresponding to $\La[i]$.
	\item morphisms: $\C^{\leq0}(i,j)=\RHom_{\La}(\La[i],\La[j])^{\leq0}=(\La[j-i])^{\leq0}=\begin{cases}\La[j-i]& (j\leq i)\\ 0& (j>i)\end{cases}$.
\end{itemize}
We conclude that $\C^{\leq0}$ can be identified with the Adams graded dg algebra $\La\otimes_kk[x]$ with $x$ in cohomological degree $-1$, and in Adams degree $1$. 
For example, if $\La=k$, it is presented by the linear dg quiver with degree $-1$ arrows between the neighboring vertices and with trivial differentials.
\[ \xymatrix{ \cdots\ar[r]&\circ\ar[r]^-{-1}&\circ\ar[r]^-{-1}&\circ\ar[r]^-{-1}&\circ\ar[r]^-{-1}&\circ\ar[r]&\cdots } \]
\end{Ex}

We end this section with the following information.
\begin{Prop}
Let $\A\to\B$ be a dg functor between dg categories with shifts. Then, the canonical maps $\A\lotimes_{\A^{\leq0}}\B^{\leq0}\to\B$ in $\rD(\A^\op\otimes\B)$ and $\B^{\leq0}\lotimes_{\A^{\leq0}}\A\to\B$ in $\rD(\B^{\leq0,\op}\otimes\A)$ are isomorphisms. Therefore, the diagrams below are commutative.
\[	\xymatrix@C=20mm{
	\rD(\A^{\leq0})\ar[r]^-{-\lotimes_{\A^{\leq0}}\B^{\leq0}}&\rD(\B^{\leq0})\\
	\rD(\A)\ar[r]^-{-\lotimes_{\A}\B}\ar@{^(->}[u]^-{\mathrm{res}}&\rD(\B)\ar@{^(->}[u]_-{\mathrm{res}} }
	\qquad
	\xymatrix@C=20mm{
	\rD(\A^{\leq0})\ar[r]^-{\RHom_{\A^{\leq0}}(\B^{\leq0},-)}&\rD(\B^{\leq0})\\
	\rD(\A)\ar[r]^-{\RHom_{\A}(\B,-)}\ar@{^(->}[u]^-{\mathrm{res}}&\rD(\B)\ar@{^(->}[u]_-{\mathrm{res}} } \]
\end{Prop}
\begin{proof}
	We only prove half of the assertions. Since $\A^{\leq0}\to\A$ is a localization, the natural map $\A\lotimes_{\A^{\leq0}}\B\to \A\lotimes_\A\B=\B$ is an isomorphism in $\rD(\A^\op\otimes\B)$. We see that the mapping cone of the morphism $\A\lotimes_{\A^{\leq0}}\B^{\leq0}\to \B$ is $\A\lotimes_{\A^{\leq0}}\B^{>0}$, which is $0$ by \ref{ubb}. This shows the commutativity of the left square.
\end{proof}

\subsection{Basic properties}
We first note a simple observation which has already been used in the proof of \ref{loc}. For each $i\in\Z$, we regard $\C^{\leq i}$ the bimodule over $\C^{\leq0}$, thus by definition, $\C^{\leq i}(C_1,C_2)=\C(C_1,C_2)^{\leq i}$.
\begin{Lem}\label{kaku}
For each $i\in\Z$, we have $\C^{\leq i}(-,?)=\C^{\leq0}(-,?[i])[-i]$ as $(\C^{\leq0},\C^{\leq0})$-bimodules.
\end{Lem}

\begin{Lem}\label{inv}
Let $\C$ be a dg category with shifts. Then for each $i\in\Z$, we have canonical isomorphisms
\[ \begin{aligned}
	&\C^{\leq i}\lotimes_{\C^{\leq0}}\C^{\leq j}\simeq\C^{\leq i+j}, \\
	&\C^{\leq j-i}\simeq\RHom_{\C^{\leq0}}(\C^{\leq i},\C^{\leq j})
	\end{aligned} \]
in $\rD((\C^{\leq0})^e)$. Therefore, the $(\C^{\leq0},\C^{\leq0})$-bimodule $\C^{\leq i}$ is invertible.
\end{Lem}
\begin{proof}
	By \ref{kaku} we have $\C^{\leq i}(-,?)=\C^{\leq0}(-,?[i])[-i]$, so $(\C^{\leq i}\lotimes_{\C^{\leq0}}\C^{\leq j})(-,?)=\C^{\leq0}(\sim,?[i])[-i]\lotimes_{\C^{\leq0}}\C^{\leq0}(-[-j],\sim)[-j]=\C^{\leq0}(-[-j],?[i])[-i-j]=\C^{\leq i+j}(-,?)$. This proves the first isomorphism. The second one is proved similarly.
\end{proof}

We also repeat \ref{per}.
\begin{Lem}\label{matakaku}
For each $i\in\Z$, we have $H^i\C\in\per\C^{\leq0}$.
\end{Lem}

\begin{Prop}\label{Res}
Let $\C$ be a dg category with shifts.
Then, the category $\thick H^0\C$ consists of all the perfect $\C^{\leq0}$-modules which are uniformly bounded. 
Therefore, every autoequivalence on $\per\C^{\leq0}$ restricts to $\thick H^0\C$.
\end{Prop}
\begin{proof}
	We know by \ref{matakaku} that $H^0\C\in\per\C^{\leq0}$, and each $H^0\C(-,C)$ is clearly uniformly bounded. To see the converse inclusion, if $M$ is uniformly bounded, then $M\lotimes_{\C^{\leq0}}\C=0$ by \ref{ubb}, thus $M\in\thick H^0\C$ by \ref{loc}.
\end{proof}

\begin{Lem}\label{back}
Let $\C$ be a dg category with shifts. Then $M^{\leq 0}\lotimes_{\C^{\leq0}}\C\xsimeq M$ for all $M\in\rD(\C)$.
\end{Lem}
\begin{proof}
	Applying $-\lotimes_{\C^{\leq0}}\C$ to the triangle $M^{\leq0}\to M\to M^{>0}$ yields
	\[ \xymatrix{ M^{\leq0}\lotimes_{\C^{\leq0}}\C\ar[r]&M\lotimes_{\C^{\leq0}}\C\ar[r]& M^{>0}\lotimes_{\C^{\leq0}}\C }. \]
	Here, the middle term is isomorphic to $M\lotimes_{\C}\C=M$ since $\C^{\leq0}\to\C$ is a localization by \ref{loc}, and the last term is $0$ by \ref{ubb}. We deduce that the first map, which identifies with $M^{\leq0}\lotimes_{\C^{\leq0}}\C\to M$, is an isomorphism.
\end{proof}

\begin{Lem}\label{Mleq0}
Let $\C$ be a pretriangulatd dg category. 
\begin{enumerate}
	\item\label{Mper} $M^{\leq 0}\in\per\C^{\leq0}$ for every $M\in\per\C$.
	\item\label{M0} $H^0M\in\thick H^0\C$ in $\rD(\C^{\leq0})$ for every $M\in\per\C$. In particular, $H^0M\in\per\C^{\leq0}$.
\end{enumerate}
\end{Lem}
\begin{proof}
	Since $\C$ is pretriangulated, each perfect $\C$-module is a direct summand of a representable $\C$-module. Therefore, if $M\in\per\C$, then $M^{\leq0}$ is a direct summand of a representable $\C^{\leq0}$-module, thus the assertion (\ref{Mper}) follows. Similarly, if $M\in\per\C$, then $H^0M\in\add H^0\C\subset\thick H^0\C$, which proves (\ref{M0}).
\end{proof}

We note the following ``Matlis duality''.
\begin{Prop}\label{matlis}
Let $\C$ be a dg category with shifts.
The multiplication map $\C^{\leq0}\to\REnd_{\C^{\leq0}}(\C^{>0})$ is an isomorphism.
\end{Prop}

We need a preparation. For each $p,q\in\Z\cup\{\pm\infty\}$ with $p\leq q$, define the bimodule $\C^{[p,q]}$ by the triangle $\C^{\leq p-1}\to\C^{\leq q}\to\C^{[p,q]}$. We also define some obvious variations like $\C^{(p,q]}$.
\begin{Lem}\label{dual}
\begin{enumerate}
\item\label{hitotsu} For each $i\in\Z$, there is a canonical isomorphism $H^{-i+1}\C\xsimeq\RHom_{\C^{\leq0}}(H^i\C,\C^{>0})[1]$ in $\rD((\C^{\leq0})^e)$.
\item\label{takusan} For each $p\geq1$ there is a canonical isomorphism $\C^{(-p,0]}\xsimeq\RHom_{\C^{\leq0}}(\C^{[1,p]},\C^{>0})$ in $\rD((\C^{\leq0})^e)$.
\end{enumerate}
\end{Lem}
\begin{proof}
	(\ref{hitotsu})  Consider the triangle
	\[ \xymatrix{ \C^{\leq i-1}(-,?)\ar[r]&\C^{\leq i}(-,?)\ar[r]& H^i\C(-,?)[-i] } \]
	in $\rD((\C^{\leq0})^e)$. By \ref{kaku}, we have $\RHom_{\C^{\leq0}}(\C^{\leq i},\C^{>0})=\RHom_{\C^{\leq0}}(\C^{\leq0}(-,-[i])[-i],\C^{>0})=\C^{>0}(-[i],-)[i]=\C^{>-i}$, so applying $\RHom_{\C^{\leq0}}(-,\C^{>0})$ to the above triangle yields 
	\[ \xymatrix{ \RHom_{\C^{\leq0}}(H^i\C[-i],\C^{>0})\ar[r]&\C^{>-i}\ar[r]&\C^{>-i+1} }, \]
	and therefore $\RHom_{\C^{\leq0}}(H^i\C[-i],\C^{>0})\simeq H^{-i+1}\C[i-1]$, which gives the assertion.
	
	(\ref{takusan})  This follows from (\ref{hitotsu}) by an immediate induction.
\end{proof}

\begin{proof}[Proof of \ref{matlis}]
	Since $\C^{>0}$ is the homotopy colimit $\hocolim_{p\geq1}\C^{[1,p]}$ of its truncations, we have an isomorphism $\RHom_{\C^{\leq0}}(\C^{>0},\C^{>0})=\RHom_{\C^{\leq0}}(\hocolim_{p\geq1}\C^{[1,p]},\C^{>0})=\holim_{p\geq1}\RHom_{\C^{\leq0}}(\C^{[1,p]},\C^{>0})$, which gives a triangle in the first row below.
	\[ \xymatrix@R=5mm{
		\RHom_{\C^{\leq0}}(\C^{>0},\C^{>0})\ar[r]&\Prod_{p\geq1}\RHom_{\C^{\leq0}}(\C^{[1,p]},\C^{>0})\ar[r]&\Prod_{p\geq1}\RHom_{\C^{\leq0}}(\C^{[1,p]},\C^{>0})\\
		\C^{\leq0}\ar[r]\ar[u]&\Prod_{p\geq1}\C^{(-p,0]}\ar[r]\ar[u]_-\rsimeq&\Prod_{p\geq1}\C^{(-p,0]}\ar[u]_-\rsimeq } \]
	On the other hand, we have a triangle in the second row, together with the canonical morphism to the first row, in which the right two maps are isomorphisms by \ref{dual}(\ref{takusan}). It follows that so is the leftmost one, as desired. 
\end{proof}

We next recall the following notion.
\begin{Def}[\cite{HaK}]
Let $\A\subset\B$ be an inclusion of dg categories.
\begin{enumerate}
	\item We say that it is {\it right essential} if the restriction functor $\per\B\to\rD(\A)$ is fully faithful.
	\item We say that it is {\it left essential} if the restriction functor $\per\B^\op\to\rD(\A^\op)$ is fully faithful.
\end{enumerate}
\end{Def}
For example, given a smooth connective dg algebra $A$, the inclusion $\D^b_\dg(A)\to\per_\dg\!A$ is left essential (see \cite{HaK}).

\begin{Prop}\label{LRess}
Let $\C$ be a dg category with shifts. The inclusion $\thick_\dg\!H^0\C\to\per_\dg\!\C^{\leq0}$ is left and right essential.
\end{Prop}
\begin{proof}
	We only prove the right essentialness. Write $\N=\thick_\dg\!H^0\C$.
	By the triangle $\C^{\leq0}\lotimes_\N\C^{\leq0}\to\C^{\leq0}\to\C$, the first term is isomorphic to $\C^{>0}[-1]$. It follows that we have canonical isomorphisms $\REnd_\N(\C^{\leq0})\xsimeq\REnd_{\C^{\leq0}}(\C^{\leq0}\lotimes_\N\C^{\leq0})\simeq\REnd_{\C^{\leq0}}(\C^{>0})$, which is isomorphic to $\C^{\leq0}$ by \ref{matlis}. 
\end{proof}

\section{Derived functor categories and Auslander-Reiten-Serre duality}\label{coherent}
Given an abelian category $\A$, for example, the category of finitely generated modules over a finite dimenisonal algebra, Auslander proposed \cite{Au66} to study the functor category $\mod\A$ of finitely presented functors over $\A$.
He established some basic properties of the functor category, 
which later lead to the development of Auslander-Reiten theory \cite{Au71,AR74}.

The aim of this section is to demonstrate that the derived category $\per\C^{\leq0}$ of the truncation plays a role as the functor category. We give an account of existence of almost split sequences, or a Serre functor in $\per\C$ along the line of Auslander-Reiten duality \cite{Au66,AR74,Au78}. Then we give a modern account in terms of cosingularity-type categories (\ref{DAR}).

%

\subsection{Classical theory for module categories}
For comparison, we start with a review of the classical Auslander-Reiten theory from a point of view of functor categories \cite{Au66} and Auslander-Reiten duality \cite{AR74}. 
Let $\La$ be a finite dimensional algebra, and let $\mod(\mod\La)$ be the functor category, that is, the category of finitely presented (contravariant) functors from $\mod\La$ to $\Mod\Z$. The fundamental result we recall is the following.
\begin{Prop}[\cite{Au66}]\label{Fun}
Let $\La$ be a finite dimensional algebra. Then $\mod\La$ (as an underlying additive category) has global dimension $\leq2$.
\end{Prop}
\begin{proof}
	For comparison, we include the well-established proof. Let $X\in\mod(\mod\La)$, so it has a finite presentation $\Hom_\La(-,B)\to\Hom_\La(-,C)\to X\to 0$ for some $B,C\in\mod\La$. By Yoneda lemma, the first morphism in the functor category is presented by a morphism $B\to C$ in $\mod\La$. Extend it to the kernel to get a left exact sequence $0\to A\to B\to C$ in $\mod\La$, which yields an exact sequence
	\[ \xymatrix{ 0\ar[r]&\Hom_\La(-,A)\ar[r]&\Hom_\La(-,B)\ar[r]&\Hom_\La(-,C)}. \]
	This gives a projective resolution of $X$ of length $2$, thus $\pd X\leq 2$ and $\gldim(\mod(\mod\La))\leq2$.
\end{proof}

We next turn to Auslander-Reiten duality. We consider some Serre subcategories of the functor category $\mod(\mod\La)$. Let $\smod\La$ be the (projectively) stable module category, and dually $\injsmod\La$ be the injectively stable module category. Then the functor category $\mod(\smod\La)$ has enough projectives $\sHom_\La(-,C)$ for $C\in\mod\La$. The injectives have two descriptions; one as $D\sHom_\La(C,-)$, and the other as $\Ext^1_\La(-,A)$ for $A,C\in\mod\La$, which are related by the Auslander-Reiten duality \cite{AR74}. 
These are summarized as an equivalence
\[ \xymatrix{ \tau_\ast\colon\mod(\smod\La)\ar[r]^-\simeq&\mod(\injsmod\La) } \]
taking $M\in\mod(\smod\La)$ to its induction along $\tau$; it is characterized as the (right) exact functor taking $\sHom_\La(-,C)$ to $\injHom_\La(-,\tau C)=D\Ext^1_\La(C,-)$.

We first observe that this equivalence induced by ($1$-)Auslander-Reiten translation comes from the $2$-Auslander-Reiten translation $\tau_2$ on $\mod(\mod\La)$. Note that since the functor category has global dimension $2$, we have a functorial $\tau_2=D\Ext^2_\M(-,\M)$ and $\tau_2^{-1}=\Ext^2_\M(D\M,-)$ for $\M=\mod\La$.
\begin{Prop}\label{AR}
Let $\La$ be a finite dimensional algebra. The diagram below is commutative.
\[ \xymatrix@R=3mm{
	\mod(\mod\La)\ar[r]^-{\tau_2}&\mod(\mod\La)\\
	\mod(\smod\La)\ar@{}[u]|-\vsubset\ar[r]^-{\tau_\ast}&\mod(\injsmod\La)\ar@{}[u]|-\vsubset } \]
\end{Prop}
\begin{proof}
	We prove $\tau_2M=M\circ \tau^{-1}$ for each $M\in\mod(\smod\La)$. Pick an exact sequence $0\to A\to B\to C\to0$ in $\mod\La$ such that
	\[ \xymatrix{ 0\ar[r]&\Hom_\La(-,A)\ar[r]&\Hom_\La(-,B)\ar[r]&\Hom_\La(-,C)\ar[r]&M\ar[r]& 0} \]
	is exact. Note that the sequence
	\[ \xymatrix{ \sHom_\La(-,A)\ar[r]&\sHom_\La(-,B)\ar[r]&\sHom_\La(-,C)\ar[r]&M\ar[r]& 0} \]
	over the stable category is also exact.
	Applying $-\lotimes_{\M}D\M$ to the first sequence, where $\M=\mod\La$, we get an exact sequence
	\[ \xymatrix{ 0\ar[r]&\tau_2M\ar[r]&D\Hom_\La(A,-)\ar[r]&D\Hom_\La(B,-)\ar[r]&D\Hom_\La(C,-)\ar[r]& 0 }. \]
	Dually as above, we also have an exact sequence
	\[ \xymatrix{ 0\ar[r]&\tau_2M\ar[r]&D\injHom_\La(A,-)\ar[r]&D\injHom_\La(B,-)\ar[r]&D\injHom_\La(C,-) }, \]
	which is isomorphic by the Auslander-Reiten duality to
	\[ \xymatrix{ 0\ar[r]&\tau_2M\ar[r]&\Ext^1_\La(\tau^{-1}(-),A)\ar[r]&\Ext^1_\La(\tau^{-1}(-),B)\ar[r]&\Ext^1_\La(\tau^{-1}(-),C) }. \]
	Now, the second and the last exact sequences show $M\circ\tau^{-1}=\tau_2M$, as desired.
\end{proof}

We give derived versions of these results.
First one is the following analogue of \ref{Fun} from the prespective that $\C^{\leq0}$ is the ``underlying additive category'' and $\per\C^{\leq0}$ is the ``functor category''.
\begin{Prop}\label{DFun}
Let $\C$ be a pretriangulated dg category. Then $\C^{\leq0}$ has global dimension $\leq2$.
\end{Prop}
\begin{proof}
This can be proved exactly sa in \ref{Fun}. We may assume that $\C$ is idempotent complete. 
Let $X\in\mod H^0\C$, so it has a finite presentation $H^0\C(-,B)\to H^0\C(-,C)\to X\to 0$ for some $B,C\in\C$. By Yoneda lemma, the first morphism in the functor category is presented by a morphism $B\to C$ in $H^0\C$. Extend it to the mapping cocone to get a triangle $A\to B\to C$ in $H^0\C$, or equivalently a triangle $\C(-A)\to\C(-,B)\to\C(-,C)$ in $\per\C$. This yields a $4$-angle, or a concatanation of two triangles, in $\per\C^{\leq0}$:
\[ \xymatrix{ \C^{\leq0}(-,A)\ar[r]&\C^{\leq0}(-,B)\ar[r]&\C^{\leq0}(-,C)\ar[r]& X}. \]
This gives a projective resolution of $X$ of length $2$. Therefore $\pd X\leq 2$ and $\gldim\C^{\leq0}\leq2$.
%
\end{proof}

We next discuss Auslander-Reiten-Serre duality, which is the main content of this section.
Note first that while every module category has the Auslander-Reiten duality, not all triangulated categories have a Serre duality. We first give a necessary and sufficient condition (\ref{dualizing}) for a triangulated category to have Serre functors in terms of the derived functor category $\per\C^{\leq0}$, which can be seen as a derived analogue of being a dualizing variety. 
Also, since the derived functor category $\per\C^{\leq0}$ does not admit Serre functors, we use the inverse dualizing bimodule of $\C^{\leq0}$ in order to exhibit certain Serre duality in $\per\C^{\leq0}$.

Under these variations, the following first main result of this section gives an analogue of \ref{AR}. Recall that a bimodule $M$ over a dg category $\A$ is called {\it invertible} if $-\lotimes_\A M$ gives an equivalence $\rD(\A)\to\rD(\A)$. 

In fact, we also give a direct description of the Serre functor on the triangulated category $H^0\C$ in terms of the inverse dualizing bimodule $\theta$ of $\C^{\leq0}$ as the following relationship shows.
\begin{Thm}\label{DAR}
Let $\C$ be a pretriangulated dg category. Suppose that the triangulated category $H^0\C$ has a Serre functor $\nu$. 
Let $\theta:=\RHom_{(\C^{\leq0})^e}(\C^{\leq0},(\C^{\leq0})^e)$ be the inverse dualizing bimodule of $\C^{\leq0}$.
\begin{enumerate}
\item\label{theta} There exists a canonical isomorphism $\theta(-,?)[2]\simeq\C^{\leq0}(-,\nu^{-1}(?)[1])$ in $\rD((\C^{\leq0})^e)$. In particular, the bimodule $\theta$ is invertible.
\item\label{cosg-type} There is a canonical isomorphism $\C\lotimes_{\C^{\leq0}}(\theta[1])^{-1}\lotimes_{\C^{\leq0}}\C\xsimeq D\C$ in $\rD(\C^e)$.
\end{enumerate}
\end{Thm}

%

We remark that these results are obtained by the perspective that the quotient $\per\C^{\leq0}\to\per\C$ is of ``consingularity type'' (see the proof of \ref{DAR}) which explains the ``shift of degree by $1$'' in some results above.

\subsection{Invertible bimodules}
We prepare the behavior of truncations of invertible bimodules.
We start with the following general observation.
\begin{Lem}\label{each}
Let $\A$ be a connective dg category, and let $M$ be an invertible bimodule over $\A$. The following are equivalent.
\begin{enumerate}
\romanenum
\item $M(-,A)\in\add\A$ for each $A\in\A$.
\item $M(A,-)\in\add\A^\op$ for each $A\in\A$.
\end{enumerate}
In this case, we say $M$ is {\rm projective on each side}.
\end{Lem}
\begin{proof}
	We only prove (i)$\Rightarrow$(ii). The assumption means that the equivalence $-\lotimes_\A M\colon\per\A\to\per\A$ takes $\add\A$ to itself. Note that the restricted functor $\add\A\to\add\A$ is in fact an equivalence. Indeed, the subcategory $\add\A$, and so it its subcategory $\add M$ since $M$ is invertible. By \cite[2.18]{AI}, we deduce that $\add\A=\add M$ and hence the claim.
	Now, the claim shows that the inverse of $M$, given by the isomorphic bimodules $\RHom_\A(M,\A)\simeq\RHom_{\A^\op}(M,\A)$, also restricts to an equivalence on $\add\A$. This shows that $\RHom_{\A^\op}(M,\A)\in\add\A$, and hence $M\in\add\A^\op$ by dualizing. 
\end{proof}

\begin{Lem}\label{-1}
Let $\C$ be a pretriangulated dg category. Suppose that $M$ is an invertible bimodule over $\C$.
\begin{enumerate}
\item\label{n} The bimodule $M^{\leq n}$ over $\C^{\leq0}$ is invertible for each $n\in\Z$.
\item\label{-n} We have an isomorphism $\RHom_{\C^{\leq0}}(M^{\leq n},\C^{\leq0})\simeq\RHom_\C(M,\C)^{\leq-n}$.
\end{enumerate}
\end{Lem}
\begin{proof}
	By $M^{\leq n}=(M[n])^{\leq0}[-n]$, we may assume that $n=0$. Indeed, this is clear for (\ref{n}), and for (\ref{-n}) note that the left-hand-side is $\RHom_{\C^{\leq0}}(M^{\leq n},\C^{\leq0})=\RHom_{\C^{\leq0}}((M[n])^{\leq0}[-n],\C^{\leq0})=\RHom_{\C^{\leq0}}((M[n])^{\leq0},\C^{\leq0})[n]$ while the right-hand-side is $\RHom_\C(M,\C)^{\leq-n}=(\RHom_\C(M,\C)[-n])^{\leq0}[n]=\RHom_\C(M[n],\C)^{\leq0}[n]$.
	
	Since $\C$ is pretriangulated, there is a dg functor $F\colon\C\to\C$ such that $H^0F\colon H^0\C\to H^0\C$ coincides with $-\lotimes_{\C}M$ up to an isomorphism.
	We claim that $-\lotimes_{\C^{\leq0}}M^{\leq0}\colon\rD(\C^{\leq0})\to\rD(\C^{\leq0})$ is nothing but the induction along $F^{\leq0}\colon\C^{\leq0}\to\C^{\leq0}$. Indeed, for each object $C$ in $\C^{\leq0}$ which is also an object of $\C$, we have $\C^{\leq0}(-,C)\lotimes_{\C^{\leq0}}M^{\leq0}=M^{\leq0}(-,C)$, which is by definition $M(-,C)^{\leq0}=\C(-,FC)^{\leq0}=\C^{\leq0}(-,FC)=\C^{\leq0}(-,F^{\leq0}C)$. This proves the claim. 
	
	Now, take a dg functor $G\colon\C\to\C$ which represents the inverse $-\lotimes_\C\RHom_\C(M,\C)$. Since the functors $F$ and $G$ are mutually inverse on $\C$, so are $F^{\leq0}$ and $G^{\leq0}$ on $\C^{\leq0}$. By the above claim, this shows $-\lotimes_{\C^{\leq0}}M^{\leq0}$ and $-\lotimes_{\C^{\leq0}}\RHom_\C(M,\C)^{\leq0}$ are mutually inverse. This proves both (\ref{n}) and (\ref{-n}).
	%
	%
\end{proof}

\begin{Lem}\label{proj-leq}
Let $\C$ be a pretriangulated dg category, and let $X$ be an invertible bimodule over $\C^{\leq0}$. Suppose that $X$ is projective on each side. Then the canonical map $X\to(\C\lotimes_{\C^{\leq0}}X\lotimes_{\C^{\leq0}}\C)^{\leq0}$ is an isomorphism.
\end{Lem}
\begin{proof}
	By \cite{HaK} we know that there is an isomorphism $X\lotimes_{\C^{\leq0}}\C\simeq\C\lotimes_{\C^{\leq0}}X\lotimes_{\C^{\leq0}}\C$ in $\rD((\C^{\leq0})^\op\otimes\C)$, thus we have to show that the canonical map $X\to (X\lotimes_{\C^{\leq0}}\C)^{\leq0}$ is an isomorphism. Since $X$ is projective as a right $\C^{\leq0}$-module, for each $C\in\C^{\leq0}$ there is $C^\prime\in\C^{\leq0}$ such that $X(-,C)\simeq\C^{\leq0}(-,C^\prime)$. It follows that $(X(-,C)\lotimes_{\C^{\leq0}}\C)^{\leq0}=\C(-,C^\prime)^{\leq0}=X(-,C)$, as desired.
\end{proof}

Recall from \ref{Res} that every autoequivalence on $\per\C^{\leq0}$ restricts to $\thick H^0\C$. It follows that every autoequivalence $F$ on $\per\C^{\leq0}$ induces an autoequivalence $\overline{F}$ on the Verdier quotient $\per\C$. If $F$ is given by an invertible bimodule $X$ over $\C^{\leq0}$, then by \cite{HaK}, the autoequivalence $\overline{F}$ on $\per\C$ is given by $\C\lotimes_{\C^{\leq0}}X\lotimes_{\C^{\leq0}}\C$. Similarly, since the quasi-inverse $F^{-1}$ of $F$ induces the one $\overline{F}^{-1}$ of $\overline{F}$, we have $(\C\lotimes_{\C^{\leq0}}X\lotimes_{\C^{\leq0}}\C)^{-1}\simeq\C\lotimes_{\C^{\leq0}}X^{-1}\lotimes_{\C^{\leq0}}\C$ for each invertible bimodule $X$ over $\C^{\leq0}$. 
\begin{Lem}\label{leq-ind}
Let $\C$ be a pretriangulated dg category, let $X$ be an invertible bimodule over $\C^{\leq0}$, and let $M$ be an invertible bimodule over $\C$. Suppose that $X$ is projective on each side.
Then the maps
	\[
	\xymatrix{ \C\lotimes_{\C^{\leq0}}-\lotimes_{\C^{\leq0}}\C\colon\Hom_{\rD((\C^{\leq0})^e)}(X,M^{\leq0})\ar@<2pt>[r]&\Hom_{\rD(\C^e)}(\C\lotimes_{\C^{\leq0}}X\lotimes_{\C^{\leq0}}\C,M)\ar@<2pt>[l]\colon(-)^{\leq0}
	} \]
	are mutually inverse, and preserves and detects isomorphisms.
\end{Lem}
\begin{proof}
	We have $X\xsimeq(\C\lotimes_{\C^{\leq0}}X\lotimes_{\C^{\leq0}}\C)^{\leq0}$ by \ref{proj-leq}, so the leftgoing map $(-)^{\leq0}$ is well-defind.
	Also, if $M$ is invertible, so is $M^{\leq0}$ by \ref{-1}. We also have $\C\lotimes_{\C^{\leq0}}M^{\leq0}\lotimes_{\C^{\leq0}}\C\xsimeq M$, which can be seen, as in \ref{-1}, describing $M$ as induced by a dg auto-quasi-equivalence on $\C$. Therefore, the rightgoing map is also well-defined.
	Being induced by a functor, the maps preserve isomorphisms.
\end{proof}

\subsection{Auslander-Reiten-Serre duality via truncation}
Recall that a {\it Serre functor} on a triangulated category $\T$ is a trianlge autoequivalence $\nu$ on $\T$ such that there is a functorial isomorphism
\[ D\Hom_{\T}(X,Y)\simeq\Hom_{\T}(Y,\nu X) \]
for each $X,Y\in\T$. For example, it is well-known \cite{Hap} that for a finite dimensional $A$ which is Iwanaga-Gorenstein, then $-\lotimes_ADA$ is Serre functor of $\per A$.
By a slight abuse of terminology, we will sometimes say that a pretriangulated dg category $\C$ has a Serre functor if the triangulated category $H^0\C$ does.

We start with elementary observations on Serre functors. 
\begin{Lem}\label{DA}
Let $\C$ be a dg category. Then the following are equivalent.
\begin{enumerate}
\romanenum
\item The triangulated category $\per\C$ has a Serre functor.
\item The bimodule $D\C$ over $\C$ is invertible.
\end{enumerate}
\end{Lem}
\begin{proof}
	(ii)$\Rightarrow$(i)  If $D\C$ is invertible, then $-\lotimes_\C D\C$ gives a Serre functor of $\per\C$.
	
	(i)$\Rightarrow$(ii)  Let $S$ be a Serre functor of $\per\C$ and consider the functors
	\[ \Hom_{\per\C}(-,S\C(-,C)) \text{ and } \Hom_{\rD(\C)}(-,D\C(C,-))|_{\per\C} \]
	on $\per\C$. Since both are isomorphic to $X\mapsto D(XC)$ by Serre duality and the $k$-dual, by Yoneda lemma, there is a morphism $S\C(-,C)\to D\C(C,-)$ in $\rD(\C)$ inducing a natural isomorphism of the functors on $\per\C$. Now, we may substitute representable modules to the functorial isomorphism, which shows that the $S\C(-,C)\to D\C(C,-)$ is an isomorphism in $\rD(\C)$.	
	
	The above claim in particular shows that $D\C(C,-)$ is perfect as a right $\C$-module. Since $\per\C^\op$ has a Serre functor $S^{-1}$, we also get that $D\C$ is pefect on the left. Since we clearly have $\C\xsimeq\REnd_\C(D\C)$ and $\C^\op\xsimeq\REnd_{\C^\op}(D\C)$, we conclude that $D\C$ is invertible.
\end{proof}	

We next relate the existence of a Serre functor and the truncation.
\begin{Lem}\label{Serre-dual}
Let $\C$ be a pretriangulated category. The following are equivalent.
\begin{enumerate}
\romanenum
\item The triangulated category $\per\C$ has a Serre functor.
\item We have $\thick H^0\C=\thick D H^0\C$ in $\rD(\C^{\leq0})$. 
\end{enumerate}
\end{Lem}
\begin{proof}
	Note first that a triangulated category $\T$ has a Serre functor if and only if $D\Hom_{\T}(T,-)\in\Mod\T$ and $D\Hom_{\T}(-,T)\in\Mod\T^\op$ are representable for each $T\in\T$.
	
	(i)$\Rightarrow$(ii)  By the above remark, we have that $DH^0\C$ is representable on each side, which immediately implies the assertion.
	
	(ii)$\Rightarrow$(i)  Suppose that $\thick DH^0\C\subset\thick H^0\C$ in $\rD(\C^{\leq0})$. In particular, the $\C^{\leq0}$-module $DH^0\C(C,-)$ is perfect for each $C\in\C$. Since $\C^{\leq0}$ is a silting subcategory of $\per\C^{\leq0}$ and $DH^0\C(C,-)$ is concentrated in degree $0$, we have triangles
	\[ \xymatrix{ X_{i+1}\ar[r]&P_i\ar[r]^-{f_i}&X_{i} } \]
	for $i\geq0$ with $X_0=DH^0\C(C,-)$, $P_i\in\add\C^{\leq0}$, and $f_i$ a right $\C^{\leq0}$-approximation \cite{AI}. In particular, we get a sequence $P_1\to P_0\to DH^0\C(C,-)$ which gives an exact sequence $H^0P_1\to H^0P_0\to DH^0\C(C,-)\to 0$ in $\Mod H^0\C$. This shows that $DH^0\C(C,-)$ is finitely presented. Since the category $\mod H^0\C$ of finitely presented $H^0\C$-modules is a Frobenius category, we see that $DH^0\C(C,-)$ is representable.
	
	Similarly, suppose that $\thick H^0\C\subset\thick DH^0\C$ in $\rD(\C^{\leq0})$. Then taking the dual shows that we have $\thick DH^0\C\subset\thick H^0\C$ in $\rD(\C^{\leq0,\op})$. By the above argument for $\C^\op$, we get that $DH^0\C(-,C)$ is representable $H^0\C^\op$-module for each $C\in\C$.
	
	We conclude by these assertions that $H^0\C$ has a Serre functor.
%
%
%
\end{proof}



We also have a characterization in terms of the truncation. Recall from \ref{matlis} that there is a canonical isomorphism $\C^{\leq0}\to\REnd_{\C^{\leq0}}(\C^{>0})$, and hence we have an equivalence $\thick_{\rD(\C^{\leq0})}\C^{>0}\simeq\per\C^{\leq0}$. We may regard (\ref{>0}) below can be seen as $\C^{\leq0}$ is a ``derived dualizing variety'' \cite{AR74}.
\begin{Lem}\label{dualizing}
Let $\C$ be a pretriangulated dg category with shifts. Then the following are equivalent.
\begin{enumerate}
\romanenum
\item\label{Serre3} The triangulated category $\per\C$ has a Serre functor.
\item\label{>0} The functor $D$ induces a duality $\per\C^{\leq0}\leftrightarrow\thick_{\rD(\C^{\leq0,\op})}\C^{>0}\simeq\per\C^{\leq0,\op}$.
\end{enumerate}
\end{Lem}
\begin{proof}
%
	(\ref{Serre3})$\Rightarrow$(\ref{>0})  We have to show that $D\C^{\leq0}(-,C)\in\thick\C^{>0}$ and $D\C^{>0}(C,-)\in\per\C^{\leq0}$. Suppose that $\per\C$ has a Serre functor. Then by \ref{DA} it is given by the dg bimodule $D\C$, which implies that there is a dg functor $\nu\colon\C\to\C$ such that $D\C(-,?)\simeq\C(?,\nu(-))$ in $\rD(\C^e)$. It follows by truncating that $D\C^{\leq0}(-,C)=\C^{\geq0}(C,\nu(-))=\C^{>0}(\tau^{-1}C,-)[1]\in\thick\C^{>0}$, and $D\C^{>0}(C,-)=\C^{<0}(-,\nu C)=\C^{\leq0}(-,\tau C)[1]\in\per\C^{\leq0}$.
	
	(\ref{>0})$\Rightarrow$(\ref{Serre3})  The assumption means $D\C^{\leq0}(-,C)\in\thick_{\rD(\C^{\leq0,\op})}\C^{>0}$ and $D\C^{\leq0}(C,-)\in\thick_{\rD(\C^{\leq0})}\C^{>0}$ for every $C\in\C$.
	It is easy to see that if $X\in\thick_{\rD(\C^{\leq0,\op})}\C^{>0}$, then $H^iX\in\mod H^0\C^\op$ for every $i\in\Z$. 
	This shows that $D H^0\C(-,C)\in\mod H^0\C^\op$ and $DH^0\C(C,-)\in\mod H^0\C$ for every $C\in\C$. It follows that $DH^0\C$ is represetable on each side, and hence $\per\C$ has a Serre functor.
\end{proof}

We next turn to the invertibility of inverse dualizing bimodules of the truncation. In what follows, we write
\[ \theta=\RHom_{\C^{\leq0}}(\C^{\leq0},(\C^{\leq0})^e), \]
the inverse dualizing bimodule of $\C^{\leq0}$.
\begin{Prop}\label{invertible}
Let $\C$ be a pretriangulated dg category which has a Serre functor. Then the inverse dualizing bimodule $\theta$ of $\C^{\leq0}$ is invertible.
\end{Prop}
\begin{proof}
	By \ref{DA}, the bimodule $D\C$ is an invertible over $\C$. Consider the triangle
	\[ \xymatrix{ (D\C)^{\leq-1}\ar[r]&D\C\ar[r]&(D\C)^{\geq0} } \]
	in $\rD((\C^{\leq0})^e)$. Note that the last term is the dual of $\C^{\leq0}$. Applying $\RHom_{\C^{\leq0}}(-,\C^{\leq0})$, we get a triangle
	\[ \xymatrix{ \RHom_{\C^{\leq0}}(D\C,\C^{\leq0})\ar[r]&\RHom_{\C^{\leq0}}((D\C)^{\leq-1},\C^{\leq0})\ar[r]&\RHom_{\C^{\leq0}}(D\C^{\leq0},\C^{\leq0})[1] }. \]
	Since $D\C\in\per\C$ by assumption, the first term vanishes by \ref{LRess} (see \cite{HaK}). By \ref{-1}, the second term is the dual of an invertible bimodule $(D\C)^{\leq-1}$ over $\C^{\leq0}$, thus it is invertible. Also, the last term is $\RHom_{\C^{\leq0}}(\C^{\leq0},(\C^{\leq0})^e)=\theta$ by \ref{vee}.
	We therefore see that $\theta$ is invertible.
\end{proof}

The next property is the relative Serre duality.
\begin{Lem}\label{rel}
Let $\C$ be a dg category with shifts such that each cohomology is finite dimensional. There exist functorial isomorphisms
\[ \xymatrix{ D\RHom_{\C^{\leq0}}(X,Y)\ar@{-}[r]^-\simeq&\RHom_{\C^{\leq0}}(Y\lotimes_{\C^{\leq0}}\theta,X) } \]
for all $X\in\thick DH^0\C$ and $Y\in\per\C^{\leq0}$.
\end{Lem}
\begin{proof}
We follow \cite[4.1]{Ke08} with care. For each $X\in\thick DH^0\C$ and $Y\in\per\C^{\leq0}$, we have the following sequence of isomorphisms
\[
\begin{aligned}
	\RHom_{\C^{\leq0}}(X,Y)&=\RHom_{(\C^{\leq0})^e}(\C^{\leq0},\cHom_k(X,Y))\\
	&=\RHom_{(\C^{\leq0})^e}(\C^{\leq0},DX\otimes Y) && \text{(since $X$ takes values in $\per k$)}\\
	&=Y\lotimes_{\C^{\leq0}}\theta\lotimes_{\C^{\leq0}}DX && \text{(by $DX\in\per\C^{\leq0,\op}$ and $Y\in\per\C^{\leq0}$)}\\
\end{aligned}
\]
Therefore, taking the dual, we get the desired isomorphism.
\end{proof}

\begin{Prop}\label{relS}
Let $\C$ be a dg category with shifts such that $\per\C$ has a Serre functor. There exist functorial isomorphisms
\[ \xymatrix{ D\RHom_{\C^{\leq0}}(X,Y)\ar@{-}[r]^-\simeq&\RHom_{\C^{\leq0}}(Y\lotimes_{\C^{\leq0}}\theta,X) } \]
for all $X\in\thick H^0\C$ and $Y\in\per\C^{\leq0}$.
\end{Prop}
\begin{proof}
	Since $\per\C$ has a Serre functor, each cohomology of $\C$ is finite dimensional. Then the assertion follows from \ref{rel} and \ref{Serre-dual}.
\end{proof}

\begin{Lem}\label{proj}
Let $\C$ be a pretriangulated dg category with a Serre functor. Then for each $C\in\C$, the $\C^{\leq0}$-module $\theta(-,C)[2]\in\rD(\C^{\leq0})$ is representable.
\end{Lem}
\begin{proof}
	By \ref{relS} we have an isomorphism $\RHom_{\C^{\leq0}}(\theta,X)=D\RHom_{\C^{\leq0}}(X,\C^{\leq0})$ for every $X\in\mod H^0\C$. Taking as in \ref{DFun} a triangle $A\to B\to C$ in $\C$ such that $\C^{\leq0}(-,A)\to\C^{\leq0}(-,B)\to\C^{\leq0}(-,C)\to X$ is a $4$-angle, we see that $\RHom_{\C^{\leq0}}(X,\C^{\leq0})$ is concentrated in degree $2$, and hence $\RHom_{\C^{\leq0}}(\theta,X)$ in degree $-2$. It follows from \ref{pds} that $\theta[2]$ is projective.
\end{proof}
	
We are now ready to prove the main result of this subsection. Recall that we denote $\N=\thick_\dg\!H^0\C\subset\per_\dg\!\C^{\leq0}$.
\begin{proof}[Proof of \ref{DAR}]
	We first prove (\ref{cosg-type}). We denote by $\theta^{-1}$ the inverse of $\theta$. Since it restricts to an equivalence of $\thick H^0\C$ by \ref{Res}, we get by \cite{HaK} a triangle in the first row of the diagram below. 
Also, dualizing the standard triangle $\C^{\leq0}\lotimes_\N\C^{\leq0}\to\C^{\leq0}\to\C$ yields a one in the second row below. 
\[ \xymatrix{
	\C^{\leq0}\lotimes_\N\theta^{-1}\lotimes_\N\C^{\leq0}\ar[r]\ar[d]^-\rsimeq& \theta^{-1}\ar[r]\ar[d]^-\rsimeq& \C\lotimes_{\C^{\leq0}}\theta^{-1}\lotimes_{\C^{\leq0}}\C\ar@{-->}[d] \\
	D\C^{\leq0}\ar[r]&D(\C^{\leq0}\lotimes_\N\C^{\leq0})\ar[r]&D\C[1]} \]
We claim that these triangles fit into a commutative diagram whose vertical maps are isomorphisms.
First, the relative Serre duality \ref{relS} reads as an isomorphism $D\C^{\leq0}\simeq\theta^{-1}$ in $\rD(\N^\op\otimes\C^{\leq0})$, and hence also $D\N\simeq \theta^{-1}|_{\N^e}$. Therefore, using left and right essentiality \ref{LRess} of $\N\to \per_\dg\!\C^{\leq0}$, we get
\[
\begin{aligned}
	\C^{\leq0}\lotimes_\N\theta^{-1}\lotimes_\N\C^{\leq0}&=\C^{\leq0}\lotimes_\N D\C^{\leq0}\lotimes_\N\C^{\leq0}\\
	&=\C^{\leq0}\lotimes_\N D\RHom_{\N^\op}(\C^{\leq0},\C^{\leq0}) &&\text{(by adjunction)}\\
	&=\C^{\leq0}\lotimes_\N D\C^{\leq0}&& \text{(by left essentiality)} \\
	&=D\RHom_\N(\C^{\leq0},\C^{\leq0}) &&\text{(by adjunction)}\\
	&=D\C^{\leq0}, &&\text{(by right essentiality)} 
\end{aligned}
\]
thus the left vertical map is an isomorphism. Similarly, we have isomorphisms
\[
\begin{aligned}
	D(\C^{\leq0}\lotimes_\N\C^{\leq0})&=\RHom_{\N^\op}(\C^{\leq0},D\C^{\leq0}) && \text{(by adjuction)}\\
	&=\RHom_{\N^\op}(\C^{\leq0},\theta^{-1})&&\text{(by \ref{relS})}\\
	&=\theta^{-1},&& \text{(by left essentiality)}
\end{aligned}
\]
thus the middle vertical map is an isomorphism. 
Now we verify that the square is commutative. The commutativity amounts to saying that the map living in the upper right and the lower left corner in the diagram below is mapped into the same cohomology class in the lower right corner. 
\[ \xymatrix{
	&\RHom_{(\C^{\leq0})^e}(\theta^{-1},D(\C^{\leq0}\lotimes_\N\C^{\leq0}))\ar[d]\\
	\RHom_{(\C^{\leq0})^e}(\C^{\leq0}\lotimes_\N\theta^{-1}\lotimes_\N\C^{\leq0},D\C^{\leq0})\ar[r]&\RHom_{(\C^{\leq0})^e}(\C^{\leq0}\lotimes_\N\theta^{-1}\lotimes_\N\C^{\leq0},D(\C^{\leq0}\lotimes_\N\C^{\leq0})) } \]
By adjunction, the diagram below is isomorphic to the identity diagram on $\RHom_\N(D\N,D\N)$, and each relevant morphism corresponds to its idendity element. This proves the commutatativity, and we therefore get the desired isomorphism $\C\lotimes_{\C^{\leq0}}\theta^{-1}\lotimes_{\C^{\leq0}}\C\xsimeq D\C[1]$. Note that this is an isomorphism between $(\C,\C)$-bimodules as $(\C^{\leq0},\C^{\leq0})$-bimodules, but since $\C^{\leq0}\to\C$ is a localization by \ref{loc}, this is indeed an isomorphism of $(\C,\C)$-bimodules. 
\end{proof}


\subsection{Inverse dualizing bimodule of the truncation}
We provide another perspective on the computation of inverse dualilzing bimodules of the truncation.
\begin{Thm}\label{Si^vee}
Let $\C$ be a pretriangulated dg category with a Serre functor, and let $M$ be an invertible bimodule over $\C$. Denote $\theta=\RHom_{(\C^{\leq0})^e}(\C^{\leq0},(\C^{\leq0})^e)$ the inverse dualizing bimodule of $\C^{\leq0}$.
\begin{enumerate}
\item\label{leq0} There is a canonical isomorphism
\[ \xymatrix{ \RHom_{(\C^{\leq0})^e}(\theta[1],\RHom_{\C}(M,\C)^{\leq1})\ar[r]^-\simeq&\RHom_{\C^e}(M,D\C)^{\leq0} } \]
\item\label{nondeg} Let $f\colon\theta[1]\to\RHom_{\C}(M,\C)^{\leq1}$ in $\rD((\C^{\leq0})^e)$ and $g\colon M\to D\C$ in $\rD(\C^e)$ be morphisms corresponding to each other under the map in (\ref{leq0}). Then $f$ is an isomorphism if and only if $g$ is.
\end{enumerate}
\end{Thm}
\begin{proof}
	(\ref{leq0})  Put $M^\ast=\RHom_\C(M,\C)$ and consider the triangle $(M^\ast)^{\leq1}\to M^\ast\to (M^\ast)^{\geq2}$ in $\rD((\C^{\leq0})^e)$. Applying $\RHom_{(\C^{\leq0})^e}(\theta[1],-)$ yields a triangle
	\[ \xymatrix{ \RHom_{(\C^{\leq0})^e}(\theta[1],(M^\ast)^{\leq1})\ar[r]&\RHom_{(\C^{\leq0})^e}(\theta[1],M^\ast)\ar[r]&\RHom_{(\C^{\leq0})^e}(\theta[1],(M^\ast)^{\geq2}) }. \]
	Now, the middle term is isomorphic to $\RHom_{\C^e}(\C\lotimes_{\C^{\leq0}}\theta[1]\lotimes_{\C^{\leq0}}\C,M^\ast)$ by adjunction, which, by taking the inverse, is isomorphic to $\RHom_{\C^e}(M,D\C)$ by \ref{DAR}(\ref{cosg-type}).
	Note next that the last term is concentrated in degree $>0$ since $\theta[1]$ is concentrated in degree $\leq1$ by \ref{proj}.
	Therefore, that the middle map is a quasi-isomorphism in degree $\leq0$, which proves (\ref{leq0}).
	
	(\ref{nondeg})  The construction in (\ref{leq0}) may be summarized as the diagram
	\[ \xymatrix{ \Hom_{\rD((\C^{\leq0})^e)}(\theta[1],\RHom_{\C}(M,\C)^{\leq1})\ar[r]&\Hom_{\rD(\C^e)}((D\C)^\ast M^\ast)\ar[r]&\Hom_{\rD(\C^e)}(M,D\C) }, \]
	where the first map is induced by $\C\lotimes_{\C^{\leq0}}-\lotimes_{\C^{\leq0}}\C$ and second one by $(-)^\ast$.
	Clearly, the second map preserves and detects isomorphisms.
	Also, since $\theta[1]$ is a shift of a projective, it follows from a shift of \ref{leq-ind} that the first map preserves and detects isomorphisms.
	We therefore obtain the desired assertion.	
\end{proof}

When a triangulated category of a pretriangulated dg category has a Serre functor $\nu$, we denote by $\tau=\nu[-1]$ the Auslander-Reiten translation.
\begin{Cor}
Let $\C$ be a pretriangulated dg category with a Serre functor. Then there is a canonical isomorphism $\theta[2](-,?)\simeq\C^{\leq0}(-,\tau^{-1}(?))$ in $\rD((\C^{\leq0})^e)$.
\end{Cor}
\begin{proof}
	Applying \ref{Si^vee} to $M=D\C$, we get an isomorphism $\theta[1]\xsimeq\RHom_\C(D\C,\C)^{\leq1}$, and hence $\theta[2]\xsimeq(\RHom_\C(D\C,\C)[1])^{\leq0}$ by \ref{kaku}. We obtain the result by noting that $\RHom_\C(D\C,\C)[1]$ is the bimodule giving the functor $\tau^{-1}$.
\end{proof}

Applying \ref{Si^vee} to $\C$ with a right Calabi-Yau structure, we obtain the following consequence.
\begin{Cor}\label{d-CY}
Suppose that $\C[d]\simeq D\C$ in $\rD(\C^e)$. Then we have $\theta[d+1]\simeq\C^{\leq-d+1}$ in $\rD((\C^{\leq0})^e)$.
\end{Cor}
\begin{proof}
	Applying \ref{Si^vee} to $M=\C[d]$, we get an isomrphism $\theta[1]\xsimeq(\C[-d])^{\leq1}$. We have $(\C[-d])^{\leq1}=\C^{\leq-d+1}[-d]$ by \ref{kaku}, which gives the result.
\end{proof}

We say that a dg category $\A$ has {\it singular left $n$-Calabi-Yau structure} if ($\A$ is not necessarily smooth and) there is an isomorphism $\RHom_{\A^e}(\A,\A^e)[n]\simeq\A$ in $\rD(\A^e)$.
\begin{Cor}
Suppose that $\C$ is right $1$-Calabi-Yau, that is, $\C\simeq D\C[-1]$ in $\rD(\C^e)$. Then $\C^{\leq0}$ has a singular left $2$-Calabi-Yau structure.
\end{Cor}
\begin{proof}
	Immediate by \ref{d-CY}.
\end{proof}

\section{Connective resolutions}\label{conn-resol}
\subsection{Definition and basic examples}
We introduce the notion of {\it connective resolutions} motivated by our localization theorem \ref{loc} and by the theory of non-commutative resolutions.

\begin{Def}\label{def}
Let $\C$ be a dg category. We say that a subcategory $\M\subset\per_\dg\!\C$ gives a {\it connective resolution} if it satisfies the following conditions.
\begin{enumerate}
\item The truncation $\M^{\leq0}$ has finite right global dimension, see \ref{gldim}.
\item The composite $\M^{\leq0}\to\M\subset\per_\dg\!\C$ is a localization.
\end{enumerate}
We refer to the truncation $\M^{\leq0}$ as a connective resolution of $\C$.
If $\M^{\leq0}$ is smooth, then we call it a {\it smooth connective resolution}.
If $H^0\M\subset\per\C$ has an additive generator $M$, we say $M$ {\it gives} a connective resolution which is $\REnd_\C(M)^{\leq0}$.
\end{Def}

Replacing $\C$ by a Morita equivalent dg category does not affect the definition. In what follows, we therefore often start from a pretriangulated (and sometimes also idempotent complete) dg category $\C$. 


\begin{Lem}
Let $\C$ be a pretriangulatd dg category. Suppose that $\M\subset\C$ gives a connective resolution. Then $H^0\M$ generates the triangulated category $\per\C$ as a thick subcategory. In particular, if $M\in\per\C$ gives a connective resolution, then $\thick M=\per\C$.
\end{Lem}
\begin{proof}
	Since the functor $\M^{\leq0}\to\C$ is a localization, the image of a generating subcategory $\{\M^{\leq0}(-,M)\mid M\in\M^{\leq0}\}\subset\per\M^{\leq0}$ generates $\per\C$. Since the image of $\M^{\leq0}(-,M)$ under the induction is just $\C(-,M)$, we see that $\M$ generates $\per\C$.
\end{proof}

We next note the left-right symmetry of connective resolutions under coherence assumptions.
\begin{Lem}
Let $\C$ be a pretriangulated dg category. Let $\M\subset\C$ be a subcategory such that $H^0\M\subset H^0\C$ is functorially finite. Then $\M$ gives a connective resolution of $\C$ if and only if $\M^\op$ gives a connective resolution of $\C^\op$.
\end{Lem}
\begin{proof}
	Since $H^0\M\subset H^0\C$ is functorially finite, the additive category $H^0\M$ is coherent and each cohomology $H^i\M$ is finitely presented on each side by \ref{cohe} below. It follows from \cite{To} that finiteness of left and right global dimensions of $\M$ is equivalent. Since being a localization is left-right symmetric, the assertion follows.
\end{proof}

We next consider smooth connective resolutions. We refer to \cite[2.2.3]{BV} for the notion of strong generators.
\begin{Lem}
Let $\C$ be a pretriangulated dg category. Suppose that $\M\subset\C$ gives a smooth connective resolution.
\begin{enumerate}
\item The additive category $H^0\M$ has an additive generator.
\item The dg category $\C$ is smooth and the triangulated category $\per\C$ has a strong generator.
\end{enumerate}
\end{Lem}
\begin{proof}
	It is well-known to experts that for each smooth dg category $\A$, its derived category $\per\A$ has a strong generator, and in particular it is Morita equivalent to a smooth dg algebra.
	Since the truncation $\M^{\leq0}$ is also connective, it follows that $H^0\M$ has an additive generator. Also, smoothness is preserved under localization, so smoothness of $\M^{\leq0}$ implies that of $\C$.
\end{proof}


Let us now turn to examples which motivated our concept. Recall that a {\it non-commutative resolution} of a Noetherian ring $R$ is $M\in\mod R$ such that $\End_R(M)$ has finite global dimension.
\begin{Prop}\label{R}
Let $R$ be a Noetherian ring. 
\begin{enumerate}
\item The category $\mod R$, regarded as a full dg subcategory of $\D^b_\dg(\mod R)$, gives a connective resolution which has global dimension $2$.
\item Moreover, each non-commuatative resolution $M\in\mod R$ such that $\proj R\subset\add M$ gives a connective resolution of $\D^b_\dg(\mod R)$.
\end{enumerate}
\end{Prop}
\begin{proof}
	(1)  Let $\M=\mod R\subset\D^b_\dg(\mod R)$. Since $\M^{\leq0}(M,N)=\RHom_R(M,N)^{\leq0}=\Hom_R(M,N)$ for each $M,N\in\mod R$, the truncation $\M^{\leq0}$ is nothing but the module category. It is well-known that such an abelian category has global dimension $2$ as a additive (or dg) category \cite{Au66}. 
	To see that $\M^{\leq0}\to\D^b_\dg(\mod R)$ is a localization, note that the induced functor $\per\M^{\leq0}\to\Db(\mod R)$ of triangulated categories is nothing but the canonical quotient functor $\Kb(\mod R)\to\Db(\mod R)$, which is certainly a Verdier quotient.
	
	(2)  Similarly as in (1), if $M\in\mod R$, then the truncation $(\add_\dg\!M)^{\leq0}$ is nothing but the additive category $\add M$, which has finite global dimension by assumption. Also, it is well-known that $\proj R\subset\add M$ implies that $\per\End_R(M)\to\Db(\mod R)$ is a Verdier quotient. These show that $M$ gives a connective resolution.
\end{proof}
	
In view of \ref{R}, the following consequence of \ref{loc} shows that the truncation $\C^{\leq0}$ should be regarded as the `Auslander algebra', or the derived category $\per\C^{\leq0}$ should be regarded as the `functor category'.
\begin{Prop}\label{C}
Let $\C$ be a pretriangulated dg category. Then $\C^{\leq0}$ is a connective resolution of $\C$ of global dimension $2$. Therefore, $\C$ itself gives a connective resolution of $\C$.
\end{Prop}
\begin{proof}
	By \ref{DFun}, the dg category $\C^{\leq0}$ has global dimension $\leq2$. The canonical map $\C^{\leq0}\to\C$ is a localization by \ref{loc}.
\end{proof}

\begin{Rem}
This shows that for a pretriangulated dg category $\C$, the canonical functor $\C^{\leq0}\to\C$ is a universal connective resolution. Indeed, if a functor $\M^{\leq0}\to\M\subset\C$ is a connective resolution, it factors uniquely through $\C^{\leq0}$ since $\M^{\leq0}$ is connective.
\end{Rem}


Our next motivating example is given by cluster tilting objects in cluster categories. Let $\Pi$ be a connective $(d+1)$-Calabi-Yau dg algebra such that $H^0\Pi$ is finite dimensional with $d\geq2$, and let $\rC=\per\Pi/\Db(\Pi)$ be the cluster category. Then $\Pi\in\rC(\Pi)$ is a $d$-cluster tilting object \cite{Am09}, see also \cite{IYa1}.
\begin{Prop}\label{ct}
Let $\C=\per_\dg\!\Pi/\D^b_\dg(\Pi)$ be the dg cluster category of $\Pi$ as above. Then $\Pi\in\C$ gives a smooth connective resolution.
\end{Prop}
In fact, we will show more generally in \ref{CT} that each cluster tilting subcategory in a triangulated category gives a connective resolution of its enhancement. We however present this statement here for motivation.
\begin{proof}
	By \cite[5.9]{IYa1}, the map $\Pi=\RHom_\Pi(\Pi,\Pi)\to\RHom_\C(\Pi,\Pi)$ induced by the quotient functor $\per_\dg\!\Pi\to\C$ is a quasi-isomorphism in degree $\leq0$. It follows that the truncation $\REnd_\C(\Pi)^{\leq0}\simeq\Pi$ is smooth and that $\REnd_\C(\Pi)^{\leq0}\to\REnd_\C(M)$ is Morita equivalent to $\Pi\to\C$, which is a localization.
\end{proof}

We also note another class of easy examples.
\begin{Prop}\label{silt}
Let $\C$ be a smooth dg category. Then each silting object $M\in\per\C$ gives a smooth connective resolution.
\end{Prop}
\begin{proof}
	Indeed, the canonical map $\REnd_\C(M)^{\leq0}\to\C$ is a Morita functor if $M\in\per\C$ is a silting object. The dg algebra $\REnd_\C(M)^{\leq0}$ has finite global dimension by \ref{smooth}.
\end{proof}

Further results shall be found later. 
The aim of this section is to give a sufficient condition for an object to give a connective resolution. In fact, we show that this can be checked at the level of triangulated categories. 
As a consequence, we prove in \ref{CT} that each $d$-cluster tilting subcategory gives a connective resolution, which is of global dimension at most $d+1$.

%
\subsection{Resolutions in triangulated categories and connective resolutions}
We first formulate resolutions in triangulated categories and relate them to connective resolutions in enhancements. 
\begin{Def}\label{LRresol}
Let $\T$ be a triangulated category, and let $\M\subset\T$ be a full additive subcategory.
\begin{enumerate}
\item A {\it right $\M$-resolution} of $X\in\T$ is a sequence of triangles
\[ \xymatrix{ X_{i+1}\ar[r]&M_i\ar[r]^-{f_i}&X_i } \]
for $i\geq0$ with $X_0=X$ and $f_i$ is a right $\M$-approximation. A right $\M$-resolution of $X$ is {\it finite} if $X_n\in\M$ for some $n\geq0$. In this case, we say $X$ has a finite right $\M$-resolution.
\item A {\it left $\M$-resolution} of $X\in\T$ is a sequence of triangles
\[ \xymatrix{ X^{i}\ar[r]^-{g_i}&M^i\ar[r]&X^{i+1} } \]
for $i\geq0$ with $X^0=X$ and $g_i$ is a left $\M$-approximation. A left $\M$-resolution of $X$ is {\it finite} if $X^n\in\M$ for some $n\geq0$. In this case, we say that $X$ has a finite left $\M$-resolution.
\end{enumerate}
\end{Def}
If $X$ has a finite right $\M$-resolution, we will often denote by
\begin{equation}\label{angle}
	\xymatrix@R=1mm@C=3mm{
	&M_{n-1}\ar[dr]\ar[rr]&&&\cdots&\ar[rr]\ar[dr]&&M_1\ar[rr]\ar[dr]&&M_0\ar[dr]&\\
	M_n\ar[ur]& & X_{n-1}\ar[ur]& &&&X_2\ar[ur]&& X_1\ar[ur]&& X }
\end{equation}
the sequence of triangles, as an $(n+2)$-angle.
Similarly, if $X$ has a finite left $\M$-resolution, we denote the resolution as an $(n+2)$-angle
\begin{equation}\label{langle}
	\xymatrix@R=1mm@C=3mm{
		&M^{0}\ar[dr]\ar[rr]&&&\cdots&\ar[rr]\ar[dr]&&M^{n-2}\ar[rr]\ar[dr]&&M^{n-1}\ar[dr]&\\
		X\ar[ur]& & X^{1}\ar[ur]& &&&X^{n-2}\ar[ur]&& X^{n-1}\ar[ur]&& M^n. }
\end{equation}

%
%

We first note a consequence of the definition. Recall that an additive category $\A$ is {\it right coherent} if the category $\mod\A$ of finitely presented (right) modules is abelian. This is equivalent to that $\A$ has weak kernels. Dually, we say $\A$ is {\it left coherent} if $\mod\A^\op$ is abelian, or equivalently $\A$ has weak cokernels. 
\begin{Lem}\label{cohe}
Let $\T$ be a triangulated category and let $\M\subset\T$ be a full additive subcategory.
\begin{enumerate}
	\item Suppose that $\M$ is contravariantly finite in $\T$. Then $\M$ is right coherent and $\M(-,C[i])$ is finitely presnted for every $i\in\Z$ and $C\in\T$.
	\item Suppose that $\M$ is covariantly finite in $\T$. Then $\M$ is left coherent and $\M(C[-i],-)$ is finitely presnted for every $i\in\Z$ and $C\in\T$.
\end{enumerate}
\end{Lem}
\begin{proof}
	We only prove (1). Since $\T$ has weak kernels, so does every contravariantly finite subcategory, and hence $\M$ is right coherent. Also, for each $C\in\T$ and $i\in\Z$, take triangles $X\to M_0\to C[i]$ and $Y\to M_1\to X$ such that the second maps are right $\M$-approximations. This yields an exact sequence $\M(-,M_1)\to\M(-,M_0)\to\M(-,C[i])\to0$, and hence $\M(-,C[i])$ is finitely presented.	
\end{proof}

We now relate this notion in triangulated categories to enhancements. The following observation demonstrates the usefulness of truncations.
\begin{Lem}\label{exact}
Let $\C$ be a dg category and let $\M\subset\per_\dg\!\C$ be a full subcategory. 
\begin{enumerate}
\item\label{exr} Suppose that $X\in\per\C$ has a finite right $H^0\M$-resolution, and let \eqref{angle} be triangles which resolve $X$. Then they yield an $(n+2)$-angle
\[	\xymatrix@R=1mm@C=2mm{
		&\M^{\leq0}(-,M_{n-1})\ar[dr]\ar[rr]&&&\cdots&\ar[rr]\ar[dr]&&\M^{\leq0}(-,M_0)\ar[dr]&\\
		\M^{\leq0}(-,M_n)\ar[ur]&& \M^{\leq0}(-,X_{n-1})\ar[ur]& &&&\M^{\leq0}(-,X_1)\ar[ur]&& \M^{\leq0}(-,X) } \]
	in $\rD(\M^{\leq0})$. In particular, we have $\M^{\leq0}(-,X)\in\per\M^{\leq0}$.
\item\label{exl} Suppose that $X\in\per\C$ has a finite left $H^0\M$-resolution, and let \eqref{langle} be triangles which resolve $X$. Then they yield an $(n+2)$-angle
\[	\xymatrix@R=1mm@C=2mm{
	&\M^{\leq0}(M^{n-1},-)\ar[dr]\ar[rr]&&&\cdots&\ar[rr]\ar[dr]&&\M^{\leq0}(M^0,-)\ar[dr]&\\
	\M^{\leq0}(M^n,-)\ar[ur]&&\M(X^{n-1},-)\ar[ur]& &&&\M^{\leq0}(X^1,-)\ar[ur]&& \M^{\leq0}(X,-) } \]
in $\rD(\M^{\leq0,\op})$. In particular, we have $\M^{\leq0}(X,-)\in\per\M^{\leq0,\op}$.
\end{enumerate}
\end{Lem}
\begin{proof}
	We only prove (1). If $X_{i+1}\to M_i\to X_i$ is a triangle in $\per\C$ with the second map a right $H^0\M$-approximation, then in the triangle $\M(-,X_{i+1})\to \M(-,M_i)\to\M(-,X_i)$, the second map is surjective in $H^0$. This shows that the induced sequence
	\[ \xymatrix{ \M^{\leq0}(-,X_{i+1})\ar[r]&\M^{\leq0}(-,M_i)\ar[r]&\M^{\leq0}(-,X_i) } \]
	has an acyclic total module, in other words, a triangle in $\rD(\M^{\leq0})$. Concatenating the approximations, we obtain the desired $(n+2)$-angle.
\end{proof}

The above technique leads to the finiteness of global dimension of left or right resolutions. 
\begin{Prop}\label{finite}
Let $\C$ be a dg category and let $\M\subset\per_\dg\!\C$ be a full subcategory.
\begin{enumerate}
\item Suppose that each object in $\per\C$ has a finite right $H^0\M$-resolution {of uniformly bounded length}. 
Then $\M^{\leq0}$ has finite right global dimension.
\item Suppose that each object in $\per\C$ has a fintie left $H^0\M$-resolution {of uniformly bounded length}. 
Then $\M^{\leq0}$ has finite left global dimension.
\end{enumerate}
\end{Prop}
\begin{proof}
	We only prove (1). Take a presentation $H^0\M(-,M_1)\to H^0\M(-,M_0)\to X\to 0$ in $\mod H^0\M$. Then the first map may be presented by a morphism $M_1\to M_0$ in $H^0\M\subset\per\C$, and extend it to a triangle $N\to M_1\to M_0$ in $\per\C$. It yields a complex
	\[ \xymatrix{ \M^{\leq0}(-,N)\ar[r]&\M^{\leq0}(-,M_1)\ar[r]&\M^{\leq0}(-,M_0) }, \]
	whose total is $X$. Since the last term lies in $\per\M^{\leq0}$ by \ref{exact}, we deduce that $X\in\per\M^{\leq0}$.
\end{proof}

We next relate resolutions in triangulated categories to localizations in enhancements.
\begin{Prop}\label{LR}
Let $\C$ be a pretriangulated category, and let $\M\subset\C$ be a dg subcategory.
\begin{enumerate}
\item Suppose that each object in $H^0\C$ has a finite right $H^0\M$-resolution. Then $\C^{\leq0}\lotimes_{\M^{\leq0}}\C\xsimeq\C$ in $\rD(\C^{\leq0,\op}\otimes\C)$.
\item Suppose that each object in $H^0\C$ has a finite left $H^0\M$-resolution. Then $\C\lotimes_{\M^{\leq0}}\C^{\leq0}\xsimeq\C$ in $\rD(\C^\op\otimes\C^{\leq0})$.
\end{enumerate}
In either case, the truncation $\M^{\leq0}\to\M$ is a localization.
\end{Prop}

This is a consequence of the following more precise version.
\begin{Lem}\label{lr}
Let $\C$ be a pretriangulated category, and let $\M\subset\C$ be a dg subcategory.
\begin{enumerate}
	\item Suppose that an object $C\in H^0\C$ has a finite right $H^0\M$-resolution. Then the canonical map $\C^{\leq0}(-,C)\lotimes_{\M^{\leq0}}\C\to\C(-,C)$ in $\rD(\C)$ is an isomorphism.
	\item Suppose that an object $C\in H^0\C$ has a finite left $H^0\M$-resolution. Then the canonical map $\C\lotimes_{\M^{\leq0}}\C^{\leq0}(C,-)\to\C(C,-)$ in $\rD(\C^\op)$ is an isomorphism.
\end{enumerate}
\end{Lem}	
\begin{proof}
	We only prove (1). If $C\in\C$ has a finite right $H^0\M$-resolution, then there exist $n\geq0$ and triangles
	\[ \xymatrix{ C_{i+1}\ar[r]&M_i\ar[r]^-{f_i}&C_i } \]
	for $0\leq i\leq n-1$ with $C_0=C$ and $C_n=:M_n\in\M$ such that $f_i$ is a right $H^0\M$-approximation.
	By \ref{exact}, such triangles yield an $(n+2)$-angle
	\[ \xymatrix{
		\M^{\leq0}(-,M_n)\ar[r]&\M^{\leq0}(-,M_{n-1})\ar[r]&\cdots\ar[r]&\M^{\leq0}(-,M_1)\ar[r]&\M^{\leq0}(-,M_0)\ar[d]\\
		&&&&\M^{\leq0}(-,C) } \]
	which yields a quasi-isomorphism from the total of the first row to $\C^{\leq0}(-,C)$.
	Applying $-\lotimes_{\M^{\leq0}}\C$, the above diagram becomes
	\[ \xymatrix{
		\C(-,M_n)\ar[r]&\C(-,M_{n-1})\ar[r]&\cdots\ar[r]&\C(-,M_1)\ar[r]&\C(-,M_0)\ar[d]\\
		&&&&\C^{\leq0}(-,C)\lotimes_{\M^{\leq0}}\C. } \]
	Since total of the first row is $\C(-,C)$, we obtain the desired isomorphism.
\end{proof}

\begin{proof}[Proof of \ref{LR}]
	The isomorphisms follow immediately from \ref{lr}. We show that $\M^{\leq0}\to\M$ is a localization.	
	In the case (1), the isomorphism $\C\lotimes_{\M^{\leq0}}\C^{\leq0}\xsimeq\C$ shows that we have a commutative diagram on the left below
	\[ \xymatrix{
		\rD(\M)\ar[r]^-\res\ar[d]^-\rsimeq_-{\RHom_\M(\C,-)}&\rD(\M^{\leq0})\ar@{^(->}[d]^-{\RHom_{\M^{\leq0}}(\C^{\leq0},-)}\\
		\rD(\C)\ar[r]^-\res&\rD(\C^{\leq0}), }
	\qquad
	\xymatrix{
		\rD(\M)\ar[r]^-\res\ar[d]^-\rsimeq_-{-\lotimes_{\M}\C}&\rD(\M^{\leq0})\ar@{^(->}[d]^-{-\lotimes_{\M^{\leq0}}\C^{\leq0}}\\
		\rD(\C)\ar[r]^-\res&\rD(\C^{\leq0}) } \]
	in which the lower horizontal functor is fully faithful by \ref{loc}. Similarly, in the case (2), we get the isomorphism $\C^{\leq0}\lotimes_{\M^{\leq0}}\C\xsimeq\C$, which shows that there is a commutative diagram on the right above.
	In either case, we see that the restriction functor $\rD(\M)\to\rD(\M^{\leq0})$ is fully faithful.
\end{proof}


Now we are ready to give the main result of this section which is a sufficient condition for a subcategory to give a connective resolution.
\begin{Thm}\label{conn}
Let $\C$ be a dg category and $\M\subset\per_\dg\!\C$ a full dg subcategory. Suppose that there is $n\geq0$ such that each object in $\per\C$ has a finite right $H^0\M$-resolution of length $\leq n$. Then $\M$ gives a connective resolution of right global dimension $\leq n+2$.
\end{Thm}
\begin{proof}
	The truncation $\M^{\leq0}$ has finite global dimension by \ref{finite}, and the map $\M^{\leq0}\to\M$ is a localization by \ref{LR}. This shows that $\M$ gives a connective resolution of $\per_\dg\!\C$.
\end{proof}

Furthermore, we give a description of the kernel of the localization functor.
\begin{Thm}\label{Kernel}
Under the setting of \ref{conn}, the kernel of the induction functor $-\lotimes_{\M^{\leq0}}\M\colon\per\M^{\leq0}\to\per\M=\per\C$ is generated as a thick subcategory by $\mod H^0\M$. Therefore, there is an equivalence
\[ \per\M^{\leq0}/\thick(\mod H^0\M)\xsimeq\per\C. \]
\end{Thm}
\begin{proof}
	This is a consequence of \ref{HsubsetK} and \ref{KsubsetH} below.
\end{proof}

\begin{Lem}\label{HsubsetK}
Suppose that every object in $\per\C$ has a finite right $H^0\M$-resolution. Then for each $X\in\mod H^0\M$, we have $X\lotimes_{\M^{\leq0}}\M=0$.
\end{Lem}
\begin{proof}
	Let $X\in\mod H^0\M$ and consider the presentation $H^0\M(-,M_1)\to H^0\M(-,M_0)\to X\to 0$ in $\mod H^0\M$. Regard a morphism $M_1\to M_0$ in $H^0\M$ as a one in $\per\M=\per\C$ and extend it to a triangle $C\to M_1\to M_0$ in $\per\C$. Since each object in $\per\C$ has a finite right $H^0\M$-resolution, we may take by \ref{exact} a sequence of triangles which gives an $n$-angle $M_n\to M_{n-1}\to\cdots\to M_3\to M_2\to C$ such that the induced sequence $\M^{\leq0}(-,M_n)\to\cdots\M^{\leq0}(-,M_2)$ has acyclic total module. This gives an $(n+2)$-angle
	\[ \xymatrix{
		\M^{\leq0}(-,M_{n})\ar[r]&\M^{\leq0}(-,M_{n-1})\ar[r]&\cdots\ar[r]&\M^{\leq0}(-,M_1)\ar[r]&\M^{\leq0}(-,M_0)\ar[d]\\
		&&&&X } \]
	in $\per\M^{\leq0}$ which is a resolution of $X$.
	Now, applying $-\lotimes_{\M^{\leq0}}\M$ computes $X\lotimes_{\M^{\leq0}}\M$ as the total of
	\[ \xymatrix{
		\M(-,M_{n})\ar[r]&\M(-,M_{n-1})\ar[r]&\cdots\ar[r]&\M(-,M_1)\ar[r]&\M(-,M_0) } \]
	which is acylic.
\end{proof}

\begin{Lem}\label{KsubsetH}
Suppose that every object in $\per\C$ has a finite right $H^0\M$-resolution. If $X\in\per\M^{\leq0}$ satisfies $X\lotimes_{\M^{\leq0}}\M=0$, then $X\in\thick(\mod H^0\M)$.
\end{Lem}
\begin{proof}
	Since $\mod H^0\M\subset\per\M^{\leq0}$ by \ref{finite}, it is enough by \cite{Ne92} to show that $X\in\Loc(\mod H^0\M)$ in $\rD(\M^{\leq0})$. Consider the triangle $\M^{\leq0}\to\M\to\M^{>0}$. Applying $X\lotimes_{\M^{\leq0}}-$, we wee by assumption that $X\simeq X\lotimes_{\M^{\leq0}}M^{>0}[-1]$.
	Now by $\M^{>0}(-,M)=\hocolim_{p\geq0}\M^{[1,p]}(-,M)$ and since each $H^i\M(-,M)=H^0\M(-,M[i])$ is finitely presented by \ref{cohe}, we see that $\M^{>0}\in\Loc H^0\M$. It follows that $X\simeq X\lotimes_{\M^{\leq0}}\M^{>0}[-1]\in\Loc(\mod H^0\M)$, as desired.
\end{proof}

Combining \ref{conn} and \ref{Kernel}, we deduce the following consequence for objects giving connective resolutions. For a connective dg algebra $A$ of finite global dimension such that $H^0A$ is right coherent and each $H^iA$ is finitely presented, its {\it cosingularity category} is the Verdier quotient $\per A/\thick(\mod H^0A)=:\cosg A$.
\begin{Cor}\label{object}
Let $\C$ be a dg category. Suppose that there exist $n\geq0$ and $M\in\per\C$ such that each object in $\per\C$ has a right $(\add M)$-resolution of length $\leq n$.
\begin{enumerate}
\item The object $M$ gives a connective resolution.
\item There is a derived Morita equivalence $\cosg_\dg\!\REnd_\C(M)^{\leq0}\xsimeq\C$.
\end{enumerate}
\end{Cor}
\begin{proof}
	The assertion (1) is \ref{conn}. Since $\REnd_\C(M)^{\leq0}$ has right coherent $0$-th cohomology and each of its cohomology is finitely presented by \ref{cohe}, the assertion (2) follows from \ref{Kernel}.
\end{proof}

\subsection{Consequences}
We collect further consequences of the main result of this section.
First, under a finiteness and condition we may drop the uniform bound of the resolution. Also under the separability condition, one has smoothness of the truncation.
\begin{Thm}\label{smconn}
Let $\C$ be a $k$-linear, $\Hom$-finite dg category and let $M\in\per\C$. If every object in $\per\C$ has a finite right $(\add M)$-resolution, then $M$ gives a connective resolution of $M$.
If moreover $E=\End_{\rD(\C)}(M)$ satisfies that $E/J_E$ is separable over $k$, then $M$ gives a smooth connective resolution of $\C$.
\end{Thm}
\begin{proof}
	By the finiteness assumption, the global dimension of $\REnd_\C(M)^{\leq0}$ is the supremum of the projective dimensions of simple $H^0\REnd_\C(M)^{\leq0}$-modules. Since there are only finitely many such modules and every simple $\REnd_\C(M)^{\leq0}$-module has finite projective dimension by \ref{exact}, we see that $\REnd_\C(M)^{\leq0}$ has finite global dimension. Therefore $M$ gives a connective resolution as in \ref{conn}.
	Under the separability assumption, the smoothness follows from \ref{converse}.
\end{proof}

We may also apply \ref{conn} to cluster tilting subcategories.
\begin{Cor}\label{CT}
Let $\C$ be a dg category. If $\M\subset\per_\dg\!\C$ is a $d$-cluster tilting subcategory for $d\geq1$, then it gives a connective resolution of $\C$ of global dimension at most $d+1$.
\end{Cor}
\begin{proof}
	Since every object in $\per\C$ has a finite left and right $H^0\M$-resolution by a $d$-cluster tilting subcategory $H^0\M\subset\per\C$, it gives a connective resolution by \ref{conn}.
	We show that $\M^{\leq0}$ has global dimension $\leq d+1$.
	Let $X\in\mod H^0\M$, and consider a presentation $H^0\M(-,M_1)\to H^0\M(-,M_0)\to X\to 0$. We may represent the morphism $H^0\M(-,M_1)\to H^0\M(-,M_0)$ by a $0$-cocycle $M_1\to M_0$ in $\M$, and extend it to a triangle $C\to M_1\to M_0$ in $\per\C$. This gives a $4$-angle
	\[ \xymatrix{ 0\ar[r]&\M^{\leq0}(-,C)\ar[r]&\M^{\leq0}(-,M_1)\ar[r]&\M^{\leq0}(-,M_0)\ar[r]& X\ar[r]& 0}. \]
	Now, since $\M\subset\per\C$ is a $d$-cluster tilting subcategory, there is a $(d-1)$-angle
	\[ \xymatrix@R=1mm@C=2mm{
		&M_{d}\ar[dr]\ar[rr]&&&\cdots&\ar[rr]\ar[dr]&&M_2\ar[dr]&\\
		M_{d+1}\ar[ur]& & C_{d-2}\ar[ur]& &&& C_1\ar[ur]&& C } \]
	in which each map $M_i\to C_{i-2}$ is a right $\M$-approximation.
	This gives a $(d-1)$-angle
	\[	\xymatrix@R=1mm@C=2mm{
		&\M^{\leq0}(-,M_{d})\ar[dr]\ar[rr]&&&\cdots&\ar[rr]\ar[dr]&&\M^{\leq0}(-,M_2)\ar[dr]&\\
		\M^{\leq0}(-,M_{d+1})\ar[ur]&& \M^{\leq0}(-,C_{d-1})\ar[ur]& &&&\M^{\leq0}(-,C_1)\ar[ur]&& \M^{\leq0}(-,C) } \]
	which shows $\M^{\leq0}(-,C)$ has projective dimension $\leq d-1$, and hence $X$ has projective dimension $\leq d+1$.
\end{proof}

\begin{Cor}
Let $k$ be a perfect field and let $\C$ be a $k$-linear, $\Hom$-finite dg category. Then each $d$-cluster tilting object $M\in\per\C$ gives a smooth connective resolution of $\C$.
\end{Cor}

An additive category is {\it finite} if it has an additive generator.
\begin{Cor}\label{fin}
Let $\T$ be a finite algebraic triangulated category enhanced by a dg category $\C$. Then an additive generator $M$ of $\T$ gives a connective resolution $A=\REnd_\C(M)^{\leq0}$ of $\C$ of global dimension $\leq2$. Moreover, the inverse dualizing bimodule of $A$ is invertible. If $k$ is perfect, then $A$ is smooth.
\end{Cor}
\begin{proof}
	That $M$ gives a connective resolution of global dimension $\leq2$ is a special case of \ref{CT} for $d=1$. Since a finite triangulated category has a Serre functor, the invertibility of the inverse dualizing bimodule follows from \ref{invertible}. Finally, the smoothness is \ref{smconn}.
\end{proof}


We also note the following consequence of being a left or right resolution.
\begin{Cor}
Let $\C$ be a dg category and $\M\subset\per_\dg\!\C$. Suppose that each object in $\per\C$ has a finite right $H^0\M$-resolution. Then the dg quotient $\C^{\leq0}/\M^{\leq0}$ lies in $\thick H^0\C$ as a right $\C^{\leq0}$-module. In particular, if $\C$ is $\Hom$-finite, then $\C^{\leq0}/\M^{\leq0}$ is proper.
\end{Cor}
\begin{proof}
	Consider the triangle
	\[ \xymatrix{ \C^{\leq0}\lotimes_{\M^{\leq0}}\C^{\leq0}\ar[r]&\C^{\leq0}\ar[r]&\C^{\leq0}/\M^{\leq0} }. \]
	By \ref{exact}, the right $\M^{\leq0}$-module $\C^{\leq0}$ is perfect, so the leftmost term lies in $\per\C^{\leq0}$. It follows that $\C^{\leq0}/\M^{\leq0}\in\per\C^{\leq0}$.
	Applying $-\lotimes_{\C^{\leq0}}\C$ to the above triangle, the first map becomes $\C^{\leq0}\lotimes_{\M^{\leq0}}\C\to\C$, which is an isomorphism by \ref{LR}. We deduce that $(\C^{\leq0}/\M^{\leq0})\lotimes_{\C^{\leq0}}\C=0$, and hence $\C^{\leq0}/\M^{\leq0}\in\thick H^0\C$ by \ref{loc}.
\end{proof}

\section{Connective resolutions over connective dg algebras}\label{rep}
We demonstrate that there are sufficiently many connective resolutions over proper dg algebras. In fact, we give a generalization of \cite{Iy03} stating that any module can be extended to a connective resolution. This leads to a realization of the singularity category $\sg\La$ as the cosingularity category of a smooth connective dg algebra \cite{TV,KY18}.

Let $A$ be a connective dg $k$-algebra. We start by preparing some notions for such dg algebras and modules. The first one is for algebras. For a ring $R$ we denote by $J_R$ the Jacobson radical.
\begin{Def}
Let $A$ be a dg algebra.
\begin{enumerate}
	\item We say that $A$ is {\it connective} if $H^iA=0$ for all $i>0$, and {\it strictly connective} if $A^i=0$ for all $i>0$.
	\item For a strictly connective dg algebra $A$, we define its {\it radical} as the two-sided dg ideal $J_A:=(\bigoplus_{i<0}A^i)\oplus\pi^{-1}(J_{H^0A})$, where $\pi\colon A^0=Z^0A\to H^0A$ is the projection.
	\item For a connective dg algebra $A$, consider the canonical quasi-isomorphism $A^{\leq0}\to A$ from the truncation $A^{\leq0}$ which is a strictly connective dg algebra. We define the {\it radical $J_A$} of $A$ together with a morphism $J_A\to A$ as the one corresponding to $J_{A^{\leq0}}\to A^{\leq0}$ under the canonical equivalence $\rD((A^{\leq0})^e)\simeq\rD(A^e)$.
\end{enumerate}
\end{Def}

The next definition is for modules over connective dg algebras. For a complex $L$, we denote by $\sup L=\sup\{ i\in\Z\mid H^iL\neq0\}$. If $L$ is bounded above, we write $\sup L<\infty$.
\begin{Def}
Let $A$ be a connective dg algebra.
\begin{enumerate}
	\item An dg $A$-module is {\it semisimple} if it is concentrated in degree $0$ and is semisimple over $H^0A$.
	\item Let $M$ be a dg $A$-module such that $\sup M=n<\infty$. The {\it sup-radical} of $M$, denoted $\srad M$, is the dg submodule $(\bigoplus_{i<n}M^i)\oplus\pi^{-1}(\rad_{H^0A}\!H^nM)\subset M^{\leq n}$, where $\pi\colon Z^nM\to H^nM$ is the projection.
\end{enumerate}
\end{Def}
Note that the sup-radical is an operation at the level of ordinary rings; one just replaces the top cohomology $H^nM$ by its radical over $H^0A$.

\medskip
Now we turn to the statement of the main result of this section.
Let $\La$ be a connective dg $k$-algebra over a field $k$. We denote by
\[ \Db(\La)=\{ X\in\rD(\La) \mid \textoplus_{i\in\Z}H^iX\in\mod k \}, \]
the derived category of $\La$-modules of finite dimensional total cohomology. 
Following \cite{Iy03}, given $M\in\Db(\La)$, we inductively define $M_{n}\in\Db(\La)$ as follows. Let $M_0=M$ and define $M_{n+1}$ to be the sup-radical of $M_n$ as a module over $\REnd_\La(M_n)^{\leq0}$. Since $\La$ is connective, $M_{n+1}\subset M_n$ is an $(\REnd_\La(M)^{\leq0},\La)$-subbimodule. Also, since $M_{n+1}$ has strictly smaller total cohomology than $M_n$ (unless $M_n=0$), we have $M_m=0$ for sufficiently large $m$.

Our main result is the following generalization of \cite{Iy03}, which shows that the truncated derived endomorphism algebra is ``quasi-hereditary''.
\begin{Thm}\label{repdim}
Let $\La$ be a connective dg algebra. Put $N=\bigoplus_{i=0}^{m-1}M_i$. Then the perfect derived category of $\Ga:=\REnd_\La(N)^{\leq0}$ has a full exceptional collection. 
In particular, $\Ga$ has finite global dimension.
\end{Thm}
Even for finite dimensional algebras, we get the following consequence.
\begin{Cor}\label{cpx}
Let $\La$ be a finite dimensional algebra and $M\in\Db(\mod\La)$. Then there exists a complex $N\in\Db(\mod\La)$ containing $M$ as a direct summand such that $\per\REnd_\La(N)^{\leq0}$ has a full exceptional collection.
\end{Cor}

Applying \ref{repdim} to $M=\La$ gives the following result, which shows the abundance of connective resolutions. 
\begin{Cor}\label{NCR}
Let $\La$ be a proper connective dg algebra. Then there is $N\in\Db(\La)$ which gives a connective resolution of $\D^b_\dg(\La)$.
\end{Cor}
\begin{proof}
	Applying \ref{repdim} to $\La$, we get $N\in\Db(\La)$ containing $\La$ as a direct summand such that $\Ga:=\REnd_\La(N)^{\leq0}$ has finite global dimension. By $\add\La\subset\add N$, we see that the functor $-\lotimes_\Ga N\colon\per\Ga\to\Db(\La)$ is a localization, and therefore $N$ gives a connective resolution $\Ga$ of $\D^b_\dg(\La)$.
\end{proof}


Now we turn to the proof of \ref{repdim}.
We start by giving the definition of {\it rejective subcategories} and {\it rejective chains} \cite{Iy05b,Iy03} for derived categories in terms of truncation.

\begin{Def}\label{rej}
Let $\A$ be a full dg subcategory of a dg category $\B$.
\begin{enumerate}
\item\label{rrej} We say that $\A\subset\B$ is a {\it right rejective subcategory} if for each $B\in\B$ there is a morphism $A\to B$ in $Z^0\B$ with $A\in\A$ which induces a quasi-isomorphism on $\A$:
\[ \xymatrix{ \A(-,A)^{\leq0}\ar[r]& \B(-,B)^{\leq0} }. \]
\item We say that $\A\subset\B$ is {\it cosemisimple} if the dg quotient $\B^{\leq0}/\A^{\leq0}$ is a semisimple category (concentrated in degree $0$).
\item A {\it right rejective chain} of a dg category $\C$ is a sequence
\[ 0=\C_0\subset\C_1\subset\cdots\subset\C_{m-1}\subset\C_m=\C \]
of full dg subcategories such that $\C_{i-1}\subset\C_i$ is a right rejective cosemisimple subcategory for every $1\leq i\leq m$.
\end{enumerate}
The notions of {\it left rejective subcategories} and {\it left rejective chains} are defined dually. For example, a subcategory $\A\subset\B$ is left rejective if for each $B\in\B$ there is a morphism $B\to A$ in $Z^0\B$ with $A\in\A$ which induces a quasi-isomorphism $\A(A,-)^{\leq0}\xsimeq\A(B,-)^{\leq0}$.
\end{Def}


Recall that an object $E$ in a triangulated category $\T$ is {\it exceptional} if $\Hom_\T(E,E[i])=0$ for $i\neq0$ and $\End_\T(E)$ is a division algebra. A {\it full exceptional collection} in $\T$ is a sequence $(E_1,\ldots,E_n)$ of exceptional objects such that $\T=\thick(\bigoplus_{i=1}^nE_i)$ and $\Hom_\T(E_i,E_j[l])=0$ if $i>j$. For example, the standard modules of a quasi-hereditary algebra forms a full exceptional collection of its derived category. The following derived analogue of \cite[2.1(2)]{Iy03} shows that a rejective chain yields a ``quasi-hereditary'' algebra.
\begin{Prop}\label{qh}
Let $\C$ be a dg categeory such that $H^0\C$ has an additive generator.
Suppose that $\C$ has a right rejective chain $0=\C_0\subset\C_1\subset\cdots\subset\C_{m-1}\subset\C_m=\C$. Then $\per\C^{\leq0}$ has a full exceptional collection.
\end{Prop}
\begin{proof}
	Consider the sequence of truncations:
	\[ 0=\C_0^{\leq0}\subset\C_1^{\leq0}\subset\cdots\subset\C_{m-1}^{\leq0}\subset\C_m^{\leq0}=\C^{\leq0}. \]
	Since $\C_{i-1}\subset\C_i$ is right rejective, the module $\C_i^{\leq0}$ is representable, thus perfect, as a right $\C_{i-1}^{\leq0}$-module. It follows that the exact sequence $\C_{i-1}^{\leq0}\to\C_i^{\leq0}\to\C_i^{\leq0}/\C_{i-1}^{\leq0}$ of dg categories gives a semi-orthogonal decomposition $\per\C_i^{\leq0}=\per\C_{i-1}^{\leq0}\perp\D_i$ where $\D_i$ is equivalent to the derived category of a semisimple category.
	Therefore, we get a semi-orthogonal decomposition $\per\C^{\leq0}=\D_1\perp\D_2\perp\cdots\perp\D_n$. If $H^0\C$ has an additive generator, then each $\D_i$ is equivalent to the derived category of a semisimple algebra, thus the above decomposition can be refined to a full exceptional collection.
\end{proof}

\begin{Rem}\label{QH}
In the language of quasi-hereditary algebras, we can reformulate the above result in terms of standard modules.
For each $0\leq i\leq m$, let $M_i$ be an additive generator of $H^0\C_i$ and put $C_i=\C(M_i,M_i)$ For brevity set $C=C_m$. Let $e_i$ be the idempotent of $C_i$ or of $C_i^{\leq0}$ corresponding to $\C_i$. Define the ``standard module'' $\Del_i$ by the sequence
\[ \xymatrix{ e_iC^{\leq0}e_{i-1}\lotimes_{C_{i-1}^{\leq0}}e_{i-1}C^{\leq0}\ar[r]&e_iC^{\leq0}\ar[r]&\Del_i }. \]
It is then easy to see that $\D_i=\thick\Del_i$, and hence $\per C^{\leq0}=\thick\Del_1\perp\cdots\perp\thick\Del_m$.
\end{Rem}

We note the computation of morphisms of the dg quotients by rejective subcategories.
\begin{Lem}\label{rejquot}
Let $\A\subset\B$ be an inclusion of dg categories, and let $\C=\B^{\leq0}/\A^{\leq0}$ be the dg quotient.
\begin{enumerate}
	\item Suppose that $\A\subset\B$ is a right rejective subcategory. For each $B\in\B$, pick a morphism $A\to B$ with $A\in\A$ satisfying the condition in \ref{rej}. Then there is a triangle in $\rD(\B^{\leq0})$:
	\[ \xymatrix{ \B^{\leq0}(-,A)\ar[r]&\B^{\leq0}(-,B)\ar[r]&\C(-,B) }. \]
	\item Suppose that $\A\subset\B$ is a left rejective subcategory. For each $B\in\B$, pick a morphism $B\to A$ with $A\in\A$ satisfying the conditions in \ref{rej}. Then for every $B^\prime\in\B$, there is a triangle in $\rD(\B^{\leq0,\op})$:
	\[ \xymatrix{ \B^{\leq0}(A,-)\ar[r]&\B^{\leq0}(B,-)\ar[r]&\C(B,-) }. \]
\end{enumerate}
\end{Lem}
\begin{proof}
	We only prove (1). Since $\A\subset\B$ is right rejective, we have an isomorphism $\A(-,A)^{\leq0}\xsimeq \B(-,B)^{\leq0}$ in $\rD(\A^{\leq0})$, which is equivalent to $\B^{\leq0}(-,A)\xsimeq\B^{\leq0}(-,B)\lotimes_{\A^{\leq0}}\B^{\leq0}$ in $\rD(\B^{\leq0})$. Since the morphism complex of $\C$ fits into a triangle $\B^{\leq0}\lotimes_{\A^{\leq0}}\B^{\leq0}\to\B^{\leq0}\to\C$, the assertion follows by substituting $B$ to the left.
\end{proof}

We now return to the setting of \ref{repdim}. 
For an object $M\in\rD(A)$, we denote by $\add_\dg\!M$ the full dg subcategory of $\D_\dg(A)$, thus it has the same objects as $\add M$, and morphism complexes are given by $\RHom_A(M,M)$.
The following derived version of \cite{Iy03} is crucial.
\begin{Prop}\label{coss}
Let $A$ be a connective dg algebra, and let $M\in\Db(A)$. Put $B=\REnd_A(M)^{\leq0}$ and $N=\srad_{B^\op}M$. 
Then, $\add_\dg\!N\subset\add_\dg(M\oplus N)$ is a cosemisimple right rejective subcategory.
\end{Prop}
\begin{proof}
	We may assume $\sup M=0$. We denote $\C=\add_\dg(M\oplus N)$ and $\C^\prime=\add_\dg\!N$. We show that for each $X\in\C$ which is indecomposable in $H^0\C$, there is a morphism $Y\to X$ with $Y\in\C^\prime$ such that the induced map
	\begin{equation}\label{app}
	\xymatrix{ \C^{\leq0}(-,Y)\ar[r]&\C^{\leq0}(-,X) }
	\end{equation}
	satisfies the following.
	\begin{enumerate}
	\item\label{2} The map \eqref{app} induces an isomorphism in degree $<0$, and induces an isomorphism $H^0\C(-,Y)\xsimeq [H^0\C^\prime](-,X)$ in degree $0$.
	\item\label{3} If $X\not\in\C^\prime$, the map \eqref{app} induces an isomorphism $H^0\C(-,Y)\xsimeq J_{H^0\C}(-,X)$ in degree $0$.
	\end{enumerate}
	Then, the assertion (\ref{2}) shows that \eqref{app} is a quasi-isomorphism on $\C^\prime$, thus $\C^\prime\subset\C$ is a right rejective subcategory. The assertion (\ref{2}) also shows that the dg quotient $\C^{\leq0}/(\C^\prime)^{\leq0}$ coincides with $H^0\C/[H^0\C^\prime]$ by \ref{rejquot}, which is semisimple by (\ref{3}).
	
	Let $X\in\C$ which is indecomposable in $H^0\C$. If $X\in\C^\prime$, then the identity map on $X$ satisfies the above conditions.
	If $X\not\in H^0\C^\prime$, define a morphism $Y\to X$ in $\rD(A)$ as follows. Let $Z=\sum_{f\in J_{H^0\C}(M,X)}\Im H^0f\subset H^0X$, where $\Im$ means the image of the morphism $H^0f\colon H^0M\to H^0X$ viewed as a morphism of $H^0A$-modules. Then let $Y$ be the homotopy pull-back in $\rD(A)$:
	\[ \xymatrix{
		X^{\leq-1}\ar[r]\ar@{=}[d]&Y\ar[d]\ar[r]&Z\ar[d]\\
		X^{\leq-1}\ar[r]&X\ar[r]&H^0X. } \]
	Since $X\in\add M$, we have $Y\in\add N$.  
	Since the mapping cone of $Y\to X$ is concentrated in degree $0$, we see that the induced map $\C(-,Y)\to\C(-,X)$ yields isomorphisms in degree $<0$, and an injection in degree $0$. It remains to consider the image $F$ of $H^0\C(-,Y)\hookrightarrow H^0\C(-,X)$. There are inclusions $F\subset [H^0\C^\prime](-,X)\subset J_{H^0\C}(-,X)\subset F$. Indeed, we have the first one by $Y\in\C^\prime$, the second one by $X\not\in\C^\prime$, and the thrid one by construction. This proves both (\ref{2}) and (\ref{3}).
\end{proof}

Now we are ready to prove the main result of this section.
\begin{proof}[Proof of \ref{repdim}]
By construction, we have a sequence $0=M_m\to M_{m-1}\to\cdots\to M_1\to M_0=M$. For each $0\leq i\leq m$ put $N_i=\bigoplus_{j=1}^iM_{m-j}$ so that we have $0=N_0\subset N_1\subset N_2\subset\cdots\subset N_{m-1}\subset N_m=N$, where each inclusion is a direct summand.
By \ref{coss}, we get a right rejective chain
\[ 0=\add_\dg\!N_0\subset\add_\dg\!N_1\subset\cdots\subset\add_\dg\!N_{m-1}\subset\add_\dg\!N_m=\add_\dg\!N, \]
thus by \ref{qh}, the algebra $\Ga=\REnd_\La(N)^{\leq0}$ has a full exceptional collection.
\end{proof}

As another application of \ref{repdim}, we obtain the following result which allows us to realize the (dg) singularity category as a (dg) cosingularity category. In particular, we deduce that every dg singularity category of a proper connective dg algebra has a connective resolution.
\begin{Cor}\label{sg-cosg}
Let $\La$ be a proper connective dg algebra such that $\La/J_\La$ is separable over $k$. Then there exists a smooth connective dg algebra $B$ such that $H^0B$ is finite dimensional and $\cosg_\dg\!B\simeq\sg_\dg\!\La$.
Therefore, $B$ is a connective resolution of the dg singularity category.
\end{Cor}
\begin{proof}
Pick $M\in\Db(\La)$ which contains $\La$ as a direct summand and gives a connective resolution of $\D^b_\dg(\La)$, and put $\Ga=\REnd_\La(M)^{\leq0}$. 
Letting $B$ be the dg quotient of $\Ga$ by $e$ and $\K$ the kernel of $\per_\dg\!\Ga\xrightarrow{(-)e}\D^b_\dg(\La)$, we have a commutative diagram below in which the rows and columns are exact sequences of dg categories, see for example \cite{TV,KY18}.
\[ \xymatrix{
	&\K\ar[d]\ar@{-}[r]^-\simeq&\D^b_\dg(B)\ar[d]\\
	\per_\dg\!\La\ar[r]\ar@{=}[d]&\per_\dg\!\Ga\ar[r]\ar[d]&\per_\dg\!B\ar[d]\\
	\per_\dg\!\La\ar[r]&\D^b_\dg(\La)\ar[r]&\bullet } \]
This shows $\sg_\dg\!\La\simeq\cosg_\dg\!B$ at the lower right corner.
\end{proof}

\begin{Ex}
Let $\La=k[x]/(x^2)$ with $\deg x=0$, and let 
\[ M=\left( \xymatrix{ \La\ar[r]^-x&\La\ar[r]^-x&\cdots\ar[r]^-x&\La\ar[r]^-x&\La } \right) \]
be an $(n+1)$-term complex with the top term in degree $0$.
To compute the radical of $M$ over $\REnd_\La(M)^{\leq0}$, it is enough to consider $H^0M=k$ as a module over $H^0\REnd_\La(M)^{\leq0}=\End_{\rD(\La)}(M)=k[x]/(x^2)$. From this we see that the (large) radical of $M$ over the truncated derived endomorphism algebra is $k[n]=:S$.
One can show that the dg algebra $\REnd_\La(S\oplus M)^{\leq0}$ is presented by the following dg quiver, where $0\leq i\leq n$ thus there are $(n+1)$ arrows and loops.
\[ \xymatrix{ S\ar@<3pt>[r]^-a\ar@(ul,dl)@3[]_-{h_i}&M\ar@<3pt>@3[l]^-{b_i} } \qquad |a|=0,\, |b_i|=-i,\, |h_i|=-i-1,\, da=0,\, db_i=0,\, dh_i=b_ia. \]
Being a dg path algebra of a finite quiver, this is indeed smooth.
\end{Ex}

\section{Examples}\label{Examples}
We present some explicit descriptions of the truncation.
\subsection{Dynkin quivers}
Let $Q$ be a Dynkin quiver. It is well-known \cite{Hap} that the Auslaner-Reiten quiver of the derived category $\Db(\mod kQ)$ is $\Z Q$, and the derived category is equivalent to the mesh (or preprojective) category $k(\Z Q)$ as a $k$-linear category.
We introduce the dg version of the mesh category and prove that it is equivalent to the truncation $\D^b_\dg(\mod kQ)^{\leq0}$.

\begin{Def}\label{dpp}
Given a quiver $Q$, define the dg quiver as follows.
\begin{itemize}
\item vertices: pairs $(i,n)$ with $i\in Q_0$ and $n\in\Z$,
\item three kinds of arrows:
\begin{itemize}
\item $a=(a,n)\colon(i,n)\to (j,n)$ for each $a\colon i\to j$ in $Q$,
\item $a^\ast=(a^\ast,n)\colon (j,n)\to (i,n+1)$ for each $a\colon i\to j$ in $Q$,
\item $b_i=(b_i,n)\colon (i,n)\to (i,n+1)$ for each $i\in Q_0$,
\end{itemize}
\item grading: $\deg(a,n)=\deg(a^\ast,n)=0$ and $\deg(b_i,n)=-1$,
\item differential: $d(b_i,n)=\sum_{t(a)=i}(a,n+1)(a^\ast,n)-\sum_{s(a)=i}(a^\ast,n) (a,n)$, where the sum runs over the arrows of $Q$ whose target (resp. source) is $i$.
\end{itemize}
We call the dg path category associated with this dg quiver the {\it derived preprojective category}.
\end{Def}

For example, the derived preprojective category of type $A_4$ is presented as follows, where the diagonal arrows are $a$ and $a^\ast$, and the horizontal arrows are $b$ of degree $-1$.
\[ \xymatrix@!R=1mm@!C=3mm{
	\ar[r]&\circ\ar[rr]^-{-1}\ar[dr]&&\circ\ar[rr]^-{-1}\ar[dr]&&\circ\ar[rr]^-{-1}\ar[dr]&&\circ\ar[rr]^-{-1}\ar[dr]&&\circ\ar[rr]^-{-1}\ar[dr]&&\circ\\
	\circ\ar[ur]\ar[rr]^-{-1}\ar[dr]&&\circ\ar[ur]\ar[rr]^-{-1}\ar[dr]&&\circ\ar[ur]\ar[rr]^-{-1}\ar[dr]&&\circ\ar[ur]\ar[rr]^-{-1}\ar[dr]&&\circ\ar[ur]\ar[rr]^-{-1}\ar[dr]&&\circ\ar[ur]\ar[dr]\ar[r]&\\
	\ar[r]&\circ\ar[rr]^-{-1}\ar[dr]\ar[ur]&&\circ\ar[rr]^-{-1}\ar[dr]\ar[ur]&&\circ\ar[rr]^-{-1}\ar[dr]\ar[ur]&&\circ\ar[rr]^-{-1}\ar[dr]\ar[ur]&&\circ\ar[rr]^-{-1}\ar[dr]\ar[ur]&&\circ\\
	\circ\ar[ur]\ar[rr]^-{-1}&&\circ\ar[ur]\ar[rr]^-{-1}&&\circ\ar[ur]\ar[rr]^-{-1}&&\circ\ar[ur]\ar[rr]^-{-1}&&\circ\ar[ur]\ar[rr]^-{-1}&&\circ\ar[ur]\ar[r]&  } \]

Recall that the {\it $d$-Calabi-Yau completion} \cite{Ke11} of a smooth dg category $\A$ is the derived tensor category
\[ \Pi_d(\A)=\rT^\rL_\A(\RHom_{\A^e}(\A,\A^e)[d-1]) \]
which is endowed with an additional tensor grading, called the Adams grading.
%

For an Adams $\Z$-graded dg category $\A$, we denote by $\A\#\Z$ the {\it smash product}, that is, a dg category with
\begin{itemize}
	\item objects: pairs $(A,n)$ with $A\in\A$ and $n\in\Z$,
	\item morphisms: $(\A\#\Z)((A,n),(B,m))=\A(A,B)_{m-n}$, the Adams degree $m-n$ part of $\A(A,B)$.
\end{itemize}

These definitions leads to the following observation.
\begin{Lem}\label{covering}
The derived preprojective category of $Q$ is quasi-equivalent to the smash product $\Pi_2(kQ)\#\Z$ of the $2$-Calabi-Yau completion $\Pi_2(kQ)$ of the category of indecomposable projective $kQ$-modules.
\end{Lem}
\begin{proof}
	Since the $2$-Calabi-Yau completion $\Pi_2(kQ)=\rT_{kQ}^\rL(\RHom_{kQ}(D(kQ),kQ)[1])$ of $kQ$ is described by adding to $Q$ the opposite arrows of degree $0$ and the loops of degree $-1$, the assertion follows immediately from the definitions.
\end{proof}

Our result is the following description of the truncation as this derived preprojective category.
\begin{Thm}\label{mesh}
Let $Q$ be a Dynkin quiver. Then the derived preprojective category of $Q$ is quasi-equivalent to the full subcategory of the truncation $\D^b_\dg(\mod kQ)^{\leq0}$ formed by indecomposable objects.
\end{Thm}
\begin{proof}
	We denote by $\C=\D^b_\dg(\mod kQ)$. Up to summands, the truncation $\C^{\leq0}$ may be identified with the category whose objects are $n\in\Z$ corresponding to $\tau^{-n}kQ$, and morphisms are
	\[ \C^{\leq0}(n,m)=\RHom_{kQ}(\tau^{-n}kQ,\tau^{-m}kQ)^{\leq0}=\begin{cases}\theta_1^{\lotimes_{kQ}(m-n)}& (m\geq n) \\ 0& (m<n)\end{cases}, \]
	where $\theta_1=\RHom_{kQ}(D(kQ),kQ)[1]$. 
	It follows that $\C^{\leq0}$ is precisely the smash product $\Pi_2(kQ)\#\Z$ of the $2$-Calabi-Yau completion $\Pi_2(kQ)$. We therefore obtain the result by \ref{covering}.
\end{proof}

\begin{Cor}
There is a quasi-equivalence $\D^b_\dg(\mod kQ)^{\leq0}\simeq\Pi_2(\proj kQ)\#\Z$ which takes $\tau^{-1}$ to $(1)$.
\end{Cor}

Note that this also gives a conceptual proof of standardness of the derived category. Recall that a triangulated category with Auslander-Reiten sequences is {\it standard} if it is equivalent (as $k$-linear categories) to the mesh category of its Auslander-Reiten quiver.
\begin{Cor}[\cite{Rie}]
The derived category $\Db(\mod kQ)$ of a Dynkin quiver $Q$ is standard.
\end{Cor}
\begin{proof}
	Taking the $0$-th cohomology of the quasi-equivalence in \ref{mesh}, we see that the full subcategory of indecomposable objects of $\Db(\mod kQ)=H^0\D^b_\dg(\mod kQ)$ is equivalent to the $0$-th cohomology of the derived preprojective category. Since $Q$ is a tree, the latter is equivalent to the mesh category.
\end{proof}

\begin{Rem}
\begin{enumerate}
\item Under this equivalence $\C^{\leq0}\simeq\Pi_2(kQ)\#\Z$ for $\C=\D^b_\dg(\mod kQ)$, the exact sequence $\thick H^0\C\to\per\C^{\leq0}\to\per\C$ in \ref{loc} is isomorphic to $\pvd^\Z\!\Pi_2(kQ)\to\per^\Z\!\Pi_2(kQ)\to\D^b(\mod kQ)$
\item We can also confirm the computation of the inverse dualizing bimodule of $\D^b_\dg(\mod kQ)^{\leq0}$ given in \ref{DAR}(\ref{theta}) directly. Indeed, we have an isomorphism $\Pi^\vee[2]\simeq\Pi(1)$ in $\per^\Z\!\Pi^e$ for the derived preprojective algebra $\Pi=\Pi_2(kQ)$. Then translating to $\C^{\leq0}$ via \ref{mesh} gives the desired isomorphism $(\C^{\leq0})^\vee[2](-,?)\simeq\C^{\leq0}(-,\tau^{-1}(?))$ for $\C^{\leq0}$. 
\end{enumerate}
\end{Rem}

The above result \ref{mesh} leads to descriptions of various categories related to $\Db(\mod kQ)$. We present some of them.
\begin{Ex}\label{2term}
Let $\F:=kQ\ast kQ[1]\ast\cdots\ast kQ[n-1]\subset\D^b_\dg(\mod kQ)$ be the category of $n$-term complexes. Its truncation $\F^{\leq0}$ is presented by the following dg quiver for type $A_3$ and $n=2$.
\[ \xymatrix@!R=1mm@!C=3mm{
	&&3\ar[rr]^-{-1}\ar[dr]&&7\ar[dr]&&\\
	&2\ar[ur]\ar[rr]^-{-1}\ar[dr]&&5\ar[ur]\ar[rr]^-{-1}\ar[dr]&&8\ar[dr]&\\
	1\ar[rr]^-{-1}\ar[ur]&&4\ar[rr]^-{-1}\ar[ur]&&6\ar[rr]^-{-1}\ar[ur]&&9 } \]

Similarly, if $\La$ is the algebra presented by the quiver $\xymatrix{ 1\ar[r]^-x&2\ar[r]^-y&3}$ with relation $yx=0$, the category $\F=\La\ast\La[1]$ of $2$-term complexes consists of the following $8$ indecomposable objects.
\[ \xymatrix@!R=1mm@!C=3mm{
	\ar[r]&\circ\ar[rr]^-{-1}\ar[dr]&&2\ar[rr]^-{-1}\ar[dr]&&5\ar[rr]^-{-1}\ar[dr]&&\circ\ar[rr]^-{-1}\ar[dr]&&8\ar[rr]^-{-1}\ar[dr]&&\circ\\
	\circ\ar[ur]\ar[rr]^-{-1}\ar[dr]&&\circ\ar[ur]\ar[rr]^-{-1}\ar[dr]&&3\ar[ur]\ar[rr]^-{-1}\ar[dr]&&6\ar[ur]\ar[rr]^-{-1}\ar[dr]&&\circ\ar[ur]\ar[rr]^-{-1}\ar[dr]&&\circ\ar[ur]\ar[dr]\ar[r]&\\
	\ar[r]&1\ar[rr]^-{-1}\ar[ur]&&\circ\ar[rr]^-{-1}\ar[ur]&&4\ar[rr]^-{-1}\ar[ur]&&7\ar[rr]^-{-1}\ar[ur]&&\circ\ar[rr]^-{-1}\ar[ur]&&\circ } \]
One finds that the truncation $\F^{\leq0}$ is presented by the dg quiver below.
\[ \xymatrix@!R=1mm@!C=3mm{
	&2\ar[rr]^-{-1}\ar[dr]&&5\ar[dr]&&&\\
	1\ar[ur]\ar[rr]^-{-1}&&3\ar[ur]\ar[rr]^-{-1}\ar[dr]&&6\ar[rr]^-{-1}\ar[dr]&&8\\
	&&&4\ar[rr]^-{-1}\ar[ur]&&7\ar[ur]& } \]
\end{Ex}

We next note that truncations are compatible with orbit categories. Let $\C$ be a dg category, and let $F\colon\C\to\C$ be a dg functor inducing an equivalence on $\per\C$. The {\it dg orbit category} of $\C$ by $F$ is the dg category $\C/F$ with the same objects as $\C$ and with morphism complexes
\[ \C/F(A,B)=\colim\left( \xymatrix{ \disoplus_{i\geq0}\C(A,F^iB)\ar[r]&\disoplus_{i\geq0}\C(FA,F^iB)\ar[r]&\disoplus_{i\geq0}\C(F^2A,F^iB)\ar[r]&\cdots }\right),  \]
where the transition maps are induced by $F$. It follows immediately from the definition that we have the following.
\begin{Prop}\label{orbit}
There is a canonical quasi-equivalence $\C^{\leq0}/F\simeq(\C/F)^{\leq0}$.
\end{Prop}

\begin{Ex}\label{cluster}
Let $Q$ be a Dynkin quiver, and consider the $n$-cluster category $\rC_n(kQ)=\Db(\mod kQ)/\tau^{-1}[n-1]$ and its canonical enhancement $\C$. Then the \ref{orbit} shows that there is a quasi-equivalence
\[ \C^{\leq0}\simeq \D^b_\dg(\mod kQ)^{\leq0}/\tau^{-1}[n-1]. \]
Consequently, by \ref{mesh}, the truncation $\C^{\leq0}$ is desribed as the orbit category of the derived mesh category by $\tau^{-1}[n-1]$. For example, the truncation of the $2$-cluster category of type $A_2$ is given by the following dg path algebra with $5$ vertices.
\[ \xymatrix@!R=1mm@!C=3mm{
	\ar[r]&5\ar[rr]^-{-1}\ar[dr]&&2\ar[rr]^-{-1}\ar[dr]&&4\ar[rr]^-{-1}\ar[dr]&&1\ar[rr]^-{-1}\ar[dr]&&3\ar[rr]^-{-1}\ar[dr]&&5\\
	4\ar[ur]\ar[rr]^-{-1}&&1\ar[ur]\ar[rr]^-{-1}&&3\ar[ur]\ar[rr]^-{-1}&&5\ar[ur]\ar[rr]^-{-1}&&2\ar[ur]\ar[rr]^-{-1}&&4\ar[ur]\ar[r]&  } \]
By \ref{fin}, this is a smooth connective resolution of the dg cluster category. Moreover, its inverse dualizing bimodule is invertible.
\end{Ex}

\subsection{Yoneda categories}
We next discuss the truncation of the {\it Yoneda category} \cite{Ha2}. Let $\La$ be a finite dimensional algebra, and consider the (enhancement of) the Yoneda category
\[ \Y=\add\{ A[i]\mid A\in\mod\La, i\in\Z\}\subset\D^b_\dg(\mod\La). \]
This is a dg category with shifts. 
\begin{Thm}\label{yoneda-ncr}
Let $\La$ be a finite dimensional algebra. Then $\Y\subset\D^b_\dg(\mod\La)$ gives a connective resolution of global dimension at most $3$.
\end{Thm}
\begin{proof}
	Since $\Y$ has shifts, the functor $\Y^{\leq0}\to\Y$ is a localization by \ref{loc}, and the inclusion $\Y\to\D^b_\dg(\mod\La)$ is a Morita functor. Therefore, the composition $\Y^{\leq0}\to\D^b_\dg(\mod\La)$ is a localization. To see $\gldim\Y^{\leq0}\leq3$, let $X\in\mod H^0\Y$ and let $H^0\Y(-,A_1)\to H^0\Y(-,A_0)\to X\to 0$ be a presentation. Represent the first morphism by a one $A_1\to A_0$ in $Z^0\Y$, and extend it to a triangle $L\to A_1\to A_0$ in $\Db(\mod\La)$. Now by \cite[4.2, 4.4]{Ha2}, the subcategory $H^0\Y$ is functorially finite, and there is a triangle $A_3\to A_2\to L$ with $A_2,A_3\in\Y$ such that the second map is a right $H^0\Y$-approximation. It follows that we have a sequence
	\[ \xymatrix@C=3mm@R=1mm{
		\Y^{\leq0}(-,A_3)\ar[rr]&&\Y^{\leq0}(-,A_2)\ar[rr]\ar[dr]&&\Y^{\leq0}(-,A_1)\ar[rr]&&\Y^{\leq0}(-,A_0) \\
		&& & \Y^{\leq0}(-,L)\ar[ur] & && } \]
	whose total is $X$, and therefore $\pd X\leq3$.
	Dually, each $H^0\Y^\op$-module has projective dimension at most $3$.
\end{proof}

We call the truncation $\Y^{\leq0}$ the {\it derived Yoneda category} of $\La$.
We refer to our forthcoming paper for more information. Here we just give some descriptions of derived Yoneda algebras. 
\begin{Ex}
The derived Yoneda algebra of the Dynkin quiver of type $A_2$ is presented by the following (identical) dg path algebras.
\[ \xymatrix@!R=1mm@!C=3mm{
	\ar[r]&3\ar[rr]^-{-1}\ar[dr]&&2\ar[rr]^-{-1}\ar[dr]&&1\ar[rr]^-{-1}\ar[dr]&&3\\
	2\ar[ur]\ar[rr]^-{-1}&&1\ar[ur]\ar[rr]^-{-1}&&3\ar[ur]\ar[rr]^-{-1}&&2\ar[ur]\ar[r]&  }
\qquad
\xymatrix@!R=1mm@!C=3mm{
	&2\ar@<2pt>[dr]^-x\ar@<2pt>[dl]^-{y}&  \\
	1\ar@<2pt>[ur]^-x\ar@<2pt>[rr]^-{y}&&3\ar@<2pt>[ll]^-x\ar@<2pt>[ul]^-{y} } 
\qquad
\xymatrix@R=1mm{
	|x|=0, |y|=-1,\\
	dy=x^2} 
\]

For linearly oriented type $A_3$, the dg quiver is as follows.
\[ \xymatrix@!R=1mm@!C=3mm{
	\ar[r]&6\ar[rr]^-{-1}\ar[dr]&&3\ar[rr]^-{-1}\ar[dr]&&1\ar[rr]^-{-1}\ar[dr]&&4\ar[rr]^-{-1}\ar[dr]&&6\ar[rr]^-{-1}\ar[dr]&&3\\
	5\ar[ur]\ar[rr]^-{-1}\ar[dr]&&2\ar[ur]\ar[rr]^-{-1}\ar[dr]&&5\ar[ur]\ar[rr]^-{-1}\ar[dr]&&2\ar[ur]\ar[rr]^-{-1}\ar[dr]&&5\ar[ur]\ar[rr]^-{-1}\ar[dr]&&2\ar[ur]\ar[dr]\ar[r]&\\
	\ar[r]&1\ar[rr]^-{-1}\ar[ur]&&4\ar[rr]^-{-1}\ar[ur]&&6\ar[rr]^-{-1}\ar[ur]&&3\ar[rr]^-{-1}\ar[ur]&&1\ar[rr]^-{-1}\ar[ur]&&4 } \]
\end{Ex}

%

\end{document}